\pdfoutput=1

\documentclass[a4paper,oneside,reqno]{amsart}

\usepackage[english]{babel}
\usepackage[utf8]{inputenc}			
\usepackage[scaled]{helvet}			
\usepackage{courier}						
\usepackage{eulervm}							
\usepackage[bb=boondox,bbscaled=1.05,scr=dutchcal]{mathalfa}	%
\usepackage{dsfont}
\normalfont

\usepackage{tabularx,longtable,booktabs,multirow}			
\usepackage{caption,subcaption}			
\usepackage{fullpage}				    		
\usepackage[english]{varioref}			
\usepackage[dvipsnames]{xcolor}			
\usepackage{hyperref}								
\usepackage{bookmark}								
\usepackage{textcomp}
\usepackage{rotating,afterpage,pdflscape}
\usepackage{pbox}

\usepackage[textsize=footnotesize]{todonotes}
\presetkeys%
{todonotes}%
{color=Apricot}{}%
\usepackage{graphicx,float}		
\usepackage{tikz}							
\usepackage{tikz-cd}					
\usetikzlibrary{
  arrows,decorations.pathreplacing,decorations.markings,shapes.geometric,positioning,calc,fadings,angles
}
\tikzfading[name=inner fade, inner color=transparent!0, outer color=transparent!100]

\usepackage[style=alphabetic,maxalphanames=4,giveninits,sorting=nyt,%
maxbibnames=99,backend=bibtex]{biblatex}
\renewbibmacro{in:}{}
\usepackage{csquotes}
\usepackage[noabbrev]{cleveref}
\expandafter\def\csname ver@etex.sty\endcsname{3000/12/31}

\crefname{lemma}{lemma}{lemmata}
\Crefname{lemma}{Lemma}{Lemmata}
\crefname{subsection}{subsection}{subsections}
\Crefname{subsection}{Subsection}{Subsections}
\crefname{conjecture}{conjecture}{conjectures}
\Crefname{conjecture}{Conjecture}{Conjectures}

\usepackage{amsmath,amssymb,amsthm,mathtools}
\usepackage{bm,braket,faktor,stmaryrd}

\numberwithin{equation}{section}		

 \definecolor{webbrown}{rgb}{0.65, 0.16, 0.16}

 \hypersetup{
 colorlinks=true,linktocpage=true, pdfstartpage=1, pdfstartview=FitV,breaklinks=true, pdfpagemode=UseNone, pageanchor=true, pdfpagemode=UseOutlines,plainpages=false, bookmarksnumbered, bookmarksopen=true, bookmarksopenlevel=1,hypertexnames=true, pdfhighlight=/O,urlcolor=webbrown,linkcolor=RoyalBlue, citecolor=ForestGreen}

\newcommand{\de}{\partial}
\newcommand{\bbraket}[1]{\llbracket #1 \rrbracket}

\newcommand{\Z}{\mathbb{Z}}
\newcommand{\Q}{\mathbb{Q}}
\newcommand{\R}{\mathbb{R}}
\newcommand{\C}{\mathbb{C}}
\renewcommand{\P}{\mathbb{P}}

\newcommand{\iu}{\mathrm{i}}
\newcommand{\Id}{\mathrm{Id}}
\newcommand{\Ai}{\mathrm{Ai}}

\newcommand{\hAi}{\widehat{\mathrm{Ai}}}
\newcommand{\tAi}{\widetilde{\mathrm{Ai}}}

\newcommand{\mc}[1]{\mathcal{#1}}
\newcommand{\ms}[1]{\mathsf{#1}}

\DeclareMathOperator*{\Res}{Res}
\DeclareMathOperator{\End}{End}

\theoremstyle{plain}
\newtheorem{theorem}{Theorem}[section]
\newtheorem{proposition}[theorem]{Proposition}

\newtheorem{lemma}[theorem]{Lemma}

\theoremstyle{definition}
\newtheorem{definition}[theorem]{Definition}
\newtheorem{remark}[theorem]{Remark}
\newtheorem{example}[theorem]{Example}

\theoremstyle{plain}
\newtheorem{introthm}{Theorem}

\title{Shifted Witten classes and topological recursion}

\author[S. Charbonnier]{Séverin Charbonnier}
\address[S. Charbonnier]{
	Universit\'e de Paris, UMR 8243 CNRS, Institut de Recherche en Informatique Fondamentale, 75205 Paris Cedex 13, France %
}
\email{charbonnier@irif.fr}

\author[N. K. Chidambaram]{Nitin Kumar Chidambaram}
\address[N. K. Chidambaram]{
	Max Planck Institut f\"ur Mathematik, Vivatsgasse 7, 53111 Bonn, Germany \newline
	\newline
	 School of Mathematics, University of Edinburgh, James Clerk Maxwell Building, Peter Guthrie Tait Rd, Edinburgh EH9 3FD, U.K.}
\email{nitin.chidambaram@ed.ac.uk}

\author[E. Garcia-Failde]{Elba Garcia-Failde}
\address[E. Garcia-Failde]{
	Sorbonne Universit\'e, UMR 7586 CNRS, Institut de Math\'ematiques de Jussieu--Paris Rive Gauche, 75252 Paris, France %
}
\email{egarcia@imj-prg.fr}

\author[A. Giacchetto]{Alessandro Giacchetto}
\address[A. Giacchetto]{
	Universit\'e Paris-Saclay, UMR 3681 CNRS, CEA, Institut de Physique Th\'eorique, 91191 Gif-sur-Yvette, France %
}
\email{alessandro.giacchetto@ipht.fr}
\date{}
\subjclass[2020]{Primary 14H10, 14H70; Secondary 81R10, 34E05}
\keywords{Moduli space of curves, Witten $r$-spin class, topological recursion, $r$-KdV, $W$-constraints}

\begin{document}

\begin{abstract}
	The Witten $r$-spin class defines a non-semisimple cohomological field theory. Pandharipande, Pixton and Zvonkine studied two special shifts of the Witten class along two semisimple directions of the associated Dubrovin--Frobenius manifold using the Givental--Teleman reconstruction theorem. We show that the $R$-matrix and the translation of these two specific shifts can be constructed from the solutions of two differential equations that generalise the classical Airy differential equation. Using this, we prove that the descendant intersection theory of the shifted Witten classes satisfies topological recursion on two $1$-parameter families of spectral curves. By taking the limit as the parameter goes to zero, we prove that the descendant intersection theory of the Witten $r$-spin class can be computed by topological recursion on the $r$-Airy spectral curve. We finally show that this proof suffices to deduce Witten's $r$-spin conjecture, already proved by Faber, Shadrin and Zvonkine, which claims that the generating series of $r$-spin intersection numbers is the tau function of the $r$-KdV hierarchy that satisfies the string equation.
\end{abstract}

\maketitle

\vspace{-1cm}

\tableofcontents
\thispagestyle{empty}

\section{Introduction}

The Deligne--Mumford compactification of the moduli space of genus $g$ curves with $n$ marked points $\overline{\mc{M}}_{g,n}$ has been intensely studied in algebraic geometry, differential geometry and mathematical physics over the past few decades. In 1990, in the context of two-dimensional quantum gravity, Witten formulated his famous conjecture \cite{Wit90} that the generating series of $\psi$-classes intersection numbers on the moduli space of curves satisfies the KdV integrable hierarchy. This integrability property is very powerful and allows one to calculate the intersection numbers recursively. A year later, this celebrated conjecture was proved by Kontsevich \cite{Kon92} using a cellular decomposition of the moduli spaces and a matrix model, and the result is widely referred to as the Witten--Kontsevich theorem.

Around the same time, cohomological field theories (CohFTs) were introduced by Kontsevich and Manin \cite{KM94} to capture the formal properties of the virtual fundamental class in Gromov--Witten theory. They extend the classical notion of topological quantum field theories (TFTs), in the sense that usual linear maps induced by surfaces of genus $g$ with $ n $ boundaries from TFTs are allowed to vary cohomologically in families in CohFTs.

A special class of CohFTs are called semisimple and these have been fully classified by Givental and Teleman \cite{Tel12}. Moreover, the so-called Teleman reconstruction theorem provides an explicit algorithm to reconstruct the full higher genus theory of semisimple CohFTs starting from the genus zero theory. This reconstruction algorithm can be conveniently encoded in terms of (the local version of) a universal procedure formulated in terms of the geometry of Riemann surfaces known as the topological recursion \cite{EO07}. Topological recursion recursively constructs from certain initial data, called a (local) spectral curve, a family of differentials that encodes the CohFT correlators \cite{DOSS14}.

A famous example of a CohFT is the Witten $r$-spin class, defined for each integer $r \geq 2$. The Witten $2$-spin class turns out to be trivial, and thus it leads to no new geometry -- its associated correlators are  the Witten--Kontsevich $\psi$-classes intersection numbers. We know that they are governed by the KdV hierarchy together with the string equation. Equivalently, they satisfy the topological recursion on the (global) Airy spectral curve. For $r \geq 3$ however, the Witten $r$-spin CohFT is non-semisimple and the Teleman reconstruction theorem is not immediately applicable.

A workaround comes from the study of the genus zero sector of the Witten $r$-spin theory, which endows the underlying vector space with the structure of a Dubrovin--Frobenius manifold equivalent to the canonical Frobenius structure on the versal deformation of the $A_{r-1}$-singularity. The algebra on tangent spaces is semisimple outside the discriminant of $A_{r-1}$. In \cite{PPZ15,PPZ19}, Pandharipande--Pixton--Zvonkine studied two special semisimple directions of this Dubrovin--Frobenius manifold, which they refer to as the shifted Witten classes. Their $3$-spin construction yields Pixton's tautological relations on $\overline{\mc{M}}_{g,n}$, which extend the established Faber--Zagier relations on $\mc{M}_g$ and are conjectured to be the full set of relations in the tautological ring of $\overline{\mc{M}}_{g,n}$ \cite{Pix13}. Although the higher spin relations are implied by those for $r = 3$, the construction for general $r$ found several applications, like bounds on the Betti numbers of the tautological ring of $\mc{M}_g$ and polynomiality properties in $r$ of the Witten class. This polynomiality was employed in \cite{KLLS18,GKLS19,GZ21} to set $r=\frac12$, yielding a simplified set of relations that imply various properties of the tautological ring of the moduli space of curves.

As mentioned earlier, for every semisimple CohFT it is known that there is a local spectral curve whose topological recursion correlators compute the CohFT correlators. However, it is not clear which families of local spectral curves glue to a global one. This problem was studied in \cite{DNOS19} using the notion of a Dubrovin superpotential, and the authors show that shifted Witten classes come from global spectral curves.

Our goal is to further explore the two specific shifts of the Witten $r$-spin class considered in \cite{PPZ19} from the topological recursion perspective, that is to determine and study the two corresponding families of global spectral curves explicitly. We state these results here, but refrain from explaining the notation and details, instead leaving them for the bulk of the paper.

\begin{samepage}
\begin{introthm}[{Theorems \ref{thm:SC:Wspin:hat} and \ref{thm:SC:Wspin:tilde}}] \label{thm:spectral:curves}
	The descendant intersection theory of the shifted Witten classes can be computed using topological recursion on the following global spectral curves.
	\begin{enumerate}
		\item Let $V_r$ be the $r$-th Lucas polynomial of the second kind. The topological recursion correlators $\widehat{\omega}_{g,n}^{r,\epsilon}$ computed from the spectral curve on $\P^1$ given by
		\[
			x(z) = V_r(z,\epsilon) \, ,
			\qquad\quad
			y(z) = z \, ,
			\qquad\quad
			\omega_{0,2}(z_1,z_2) = \frac{dz_1 dz_2}{(z_1 - z_2)^2} \, ,
		\]
		encode descendant integrals of the $(\epsilon r \cdot v_{r-2})$-shifted Witten $r$-spin class $\widehat{W}^{r,\epsilon}$:
		\[
		\begin{split}
			&
			\widehat{\omega}_{g,n}^{r,\epsilon}(z_1,\dots,z_n) = \\
			&
			=
			(-r)^{g-1}
			\sum_{a_1,\dots,a_n = 0}^{r-2}
			\int_{\overline{\mc{M}}_{g,n}}
			\widehat{W}^{r,\epsilon}_{g,n}(v_{a_1} \otimes\cdots\otimes v_{a_n})
			\prod_{i=1}^n \sum_{k_i \ge 0} \psi_i^{k_i} d\hat{\xi}^{k_i,a_i}(z_i) \, ,
		\end{split}
		\]
		where $d\hat{\xi}^{k_i,a_i}(z_i)$ are explicit differentials.

		\item The topological recursion correlators $\widetilde{\omega}_{g,n}^{r,\epsilon}$ computed from the spectral curve on $\P^1$  given by
		\[
			x(z) = z^r - r \epsilon z\,,
			\qquad\quad
			y(z) = z\,,
			\qquad\quad
			\omega_{0,2}(z_1,z_2) = \frac{dz_1 dz_2}{(z_1 - z_2)^2} \, .
		\]
		encode descendant integrals of the $(\epsilon r\cdot v_1)$-shifted Witten $r$-spin class $\widetilde{W}^{r,\epsilon}$:
		\[
		\begin{split}
			&
			\widetilde{\omega}_{g,n}^{r,\epsilon}(z_1,\dots,z_n) = \\
			&
			=
			(-r)^{g-1}
			\sum_{a_1,\dots,a_n = 0}^{r-2}
			\int_{\overline{\mc{M}}_{g,n}}
			\widetilde{W}^{r,\epsilon}_{g,n}(v_{a_1} \otimes\cdots\otimes v_{a_n})
			\prod_{i=1}^n \sum_{k_i \ge 0} \psi_i^{k_i} d\tilde{\xi}^{k_i,a_i}(z_i) \, ,
		\end{split}
		\]
		where $d\tilde{\xi}^{k_i,a_i}(z_i)$ are explicit differentials. 
	\end{enumerate}
\end{introthm}
\end{samepage}

Part (2) of \cref{thm:spectral:curves} can be deduced from the general result of \cite{DNOS19}  using the  theory of Dubrovin's superpotential. Part (1), on the other hand, is new (see \cref{rem:DS} for more details). We also note here that the properties of the Lucas polynomials of the second kind explain the appearance of the trigonometric functions in the topological field theory and quantum product calculated in \cite{PPZ19} for the $ v_{r-2} $-shift. 

Our method of proof is completely different to the one of \cite{DNOS19} and thus we believe that there is a strong motivation to do our analysis. Before presenting this motivation and summarising the different perspectives that already exist in the literature, let us point out two immediate corollaries of the above theorem. 

First of all, by taking the limit of the differentials of the topological recursion as $ \epsilon \to 0 $, we prove (\cref{thm:limit}) that the correlators of the $r$-spin CohFT satisfy the Bouchard--Eynard topological recursion for the $r$-Airy spectral curve, a result that was first proved in \cite{DNOS19}. Then, using the equivalence of the Bouchard--Eynard topological recursion with certain $ W $-algebra constraints proved in \cite{BBCCN23}, and comparing these $ W $-constraints with the ones obtained in \cite{AvM92}, we deduce Witten's $r$-spin conjecture (\cref{thm:Witten}) that was first proved by Faber--Shadrin--Zvonkine \cite{FSZ10}. Witten's $r$-spin conjecture is a very powerful result which claims that $r$-spin intersection numbers are governed by the $r$-KdV hierarchy, and can be used to compute the descendant integrals recursively.

\addtocontents{toc}{\protect\setcounter{tocdepth}{1}}
\subsection*{Our motivation and method of proof}

One of our initial motivations for this paper was to explore if we could recover Witten's $r$-spin conjecture, i.e.~Faber--Shadrin--Zvonkine theorem, from topological recursion. In addition, we wanted to explicitly investigate all the elements of the identification between semisimple CohFTs and topological recursion \cite{DOSS14} in the case of shifted Witten classes.

As mentioned earlier, the family of spectral curves corresponding to the shifted Witten class was  identified in \cite{DNOS19} where the function $ x $ of the spectral curve is the Dubrovin superpotential of the associated Dubrovin--Frobenius manifold. Here we make the identification completely explicit for the shifts along $ v_1 $ and $ v_{r-2} $, and perform all the calculations necessary to apply the result of \cite{DOSS14}. The family of spectral curves associated to the $v_{r-2}$-shift corresponds to the $r$-th Lucas polynomials of the second kind, which has interesting properties and to our knowledge has never been considered before in the context of topological recursion. 

In this paper, we develop a method to compute the $R$-matrix and translation associated to a spectral curve with simple ramifications, by identifying the objects built out of the spectral curve following the prescription of \cite{DOSS14} as integral representations of solutions to families of differential equations. For both shifted Witten classes, we find generalisations of the classical Airy differential equation
\[
	\ddot{u}(t) = t \, u(t) 
\]
in two different directions. We study the asymptotic expansions of solutions to these differential equations thoroughly and find that the $R$-matrices and translations of the Witten classes shifted along the $v_1$ and $v_{r-2}$ directions can be expressed in terms of these expansions.

\begin{introthm}[{Propositions~\ref{prop:LAasymp} and \ref{prop:hyper:airy}}] \label{thm:expansions}
	Let $r\geq 3$ be a fixed integer. 
	\begin{enumerate}
		\item
		Let $0\leq a\leq r-2$ be a fixed integer. The differential equation
		\[
		 	\ddot{u}(t) = t^{r-2} \, u(t) + \frac{a}{t} \, \dot{u}(t)
		\]
		admits solutions $\hAi_{r,a}^j$ with $j=1,\dots,r-1$, called Airy--Lucas functions. Their asymptotic expansions as $t \to \infty$ determine the hypergeometric series $\ms{B}_{r,a}(u)$ defined in \cref{eqn:B:series}. The $ R $-matrix of the Witten class shifted along the $ v_{r-2} $ direction is constructed out of the series $\ms{B}_{r,a}(u)$.

		\item
		The differential equation
		\[
			u^{(r-1)}(t) = (-1)^{r-1} \, t \, u(t) \,
		\]
		admits solutions $\tAi_{r,k}$, $k=0,\dots,r-2$, called hyper--Airy functions. Their asymptotic expansions as $t \to \infty$ determine the polynomials $P_m(r,a) $ defined in \cref{eqn:PPZ:polys}. The $ R $-matrix of the Witten class shifted along the $ v_1 $ direction is constructed out of the polynomials $ P_m(r,a) $ .
	\end{enumerate}
\end{introthm}

We find part (1) of \cref{thm:expansions} surprising because a second order differential equation is controlling the $ R $-matrix for a CohFT of rank $ r-1 $. This remarkable feature explains the existence of a closed formula for the $R$-matrix in terms of the hypergeometric series $\ms{B}_{r,a} $ for the shift along $ v_{r-2} $ as found by \cite{PPZ19}. Notice that the shifts $ v_1 $ and $ v_{r-2} $ coincide when $ r=3 $, and both differential equations considered in \cref{thm:expansions} collapse to the classical Airy differential equation. In this case, the relation between the solutions of the Airy differential equation and the  Faber--Zagier tautological relations was known \cite{BJP15}, and thus our theorem can be viewed as a generalisation of this result\footnote{ In work in progress, Rosset and Zvonkine study the $ R $-matrices constructed using solutions to the family of differential equations $u^{(r-1)}(t) = (-1)^{r-1} \, t^c \, u(t)$, with $c\in\mathbb{R}$ a real parameter. They recover the relation to the Witten $r$-spin class for $c=1$ and establish a relation to Gromov--Witten invariants of projective spaces for $c\rightarrow\infty$.}. 

In the limit $\epsilon \to 0$, the spectral curves corresponding to the shifted Witten classes go to the $r$-Airy spectral curve $y^r=x$. The $ r $-Airy spectral curve can be quantised into the $r$-Airy quantum curve \cite[section 7]{BE17}: 
\[
	\biggl( \hbar^r \frac{d^r}{dx^{r}} - x \biggr) \psi(x) = 0 \,,
\] 
which corresponds to the $r$-Airy differential equation when $\hbar \rightarrow 1$. On the other hand, the $R$-matrix and translation associated to the spectral curve corresponding to the shifted Witten $r$-spin class $\widetilde{W}^{r,\epsilon}$ are expressed in terms of solutions of the $(r-1)$-Airy differential equation.

The asymptotic expansions we consider here are useful beyond the applications we had in mind in this paper. We rely strongly on these explicit computations in \cite{CGG22}, where we define the negative analogue of the Witten $r$-spin classes and prove $W$-constraints for the associated descendant integrals. We also conjecture that these $W$-constraints imply $r$-KdV, which is the negative analogue of Witten's $r$-spin conjecture, and prove it for low values of $r$ including Norbury's conjecture \cite{Nor23}.

Finally, as an application, we recover Witten's $r$-spin conjecture using topological recursion. We emphasise that our result uses the Givental--Teleman reconstruction of semisimple CohFTs \cite{Tel12}. The latter is a strong result that was not available when the proof of \cite{FSZ10} came out, and which was used in \cite{PPZ15} for the shifted Witten classes that we consider. Our result also uses that topological recursion for the generalised $r$-Airy spectral curve of \cref{thm:spectral:curves}, part (1), commutes with the limit of the parameter $\epsilon \rightarrow 0$, which we prove here for a specific class of spectral curves. A more general study of limits in topological recursion will appear in \cite{BBCKS23}.
	
\subsection*{Relation to existing work}
\label{sec:introexistingwork}

Witten's conjecture for $r=2$ has received many proofs by now, for instance \cite{Kon92,Mir07+,OP09,ABCGLW20}. The general $r$ case was only proved by Faber--Shadrin--Zvonkine in \cite{FSZ10} building on the work of Givental \cite{Giv03}. For $r > 2$, even the construction of the Witten $r$-spin class is remarkably intricate. In genus $0$, the construction was first carried out by Witten \cite{Wit93} using $r$-spin structures, which endow every point in $\overline{\mc{M}}_{g,n}$ with the additional structure of an $r$-th root of the log canonical bundle twisted at the marked points. The construction of the Witten class in higher genera was first carried out by Polishchuk--Vaintrob \cite{PV00}, and later simplified by Chiodo \cite{Chi06} using $K$-theoretic methods.

Here, we summarise the topological recursion literature around the Witten $r$-spin class, i.e. the fact that the topological recursion correlators for the $r$-Airy spectral curve compute $r$-spin intersection numbers. The result was originally conjectured by Bou\-chard and Eynard based on numerical calculations. The proofs of the following results rely on the Faber--Shadrin--Zvonkine result \cite{FSZ10}.

\begin{itemize}
	\item In \cite{DNOS19}, the authors gave the first proof that $r$-spin intersection numbers satisfy topological recursion applied to the $r$-Airy spectral curve. The theorem is an application of their main result, which claims that the Dubrovin superpotential provides a spectral curve for the CohFT associated to the corresponding Dubrovin--Frobenius manifold. The result then follows from the general consideration that limits of the classical topological recursion applied to a spectral curve with simple ramifications coincide with the corresponding Bouchard--Eynard topological recursion (see \cref{rem:limits} for subtleties related to limits).
	
	\item In \cite{BBCCN23}, higher quantum Airy structures are introduced as generalisations of Kontsevich--Soibelman quantum Airy structures by allowing differential operators of arbitrary order, instead of just quadratic ones. The authors show that the topological recursion is equivalent to certain $W$-constraints realised as higher quantum Airy structures. In particular, they  identify the higher Airy structure corresponding to the topological recursion for the $r$-Airy curve with the $W$-constraints which were proved in \cite{AvM92} to be equivalent to the $r$-KdV hierarchy together with the string equation. Then, Witten's $r$-spin conjecture implies the topological recursion result for $r$-spin intersection numbers.
	
	\item In \cite{BCEG21}, the authors built generalised Kontsevich graphs whose generating series satisfy topological recursion for a large class of spectral curves. This class of spectral curves includes the family of curves corresponding to the Witten class shifted along the $v_1$ direction. On the other hand, the generating series of graphs is identified with a generalised Kontsevich matrix model, which is showed in \cite{AvM92} to be the unique tau function of the $r$-KdV hierarchy satisfying the string equation. Again, via Witten's $r$-spin conjecture, the generating series of certain graphs is identified with the generating series of intersection numbers, which also implies the topological recursion for $r$-spin intersection numbers.
\end{itemize}

\subsection*{Acknowledgements} 

We thank Alessandro Chiodo, Sergey Shadrin, Di Yang and Dimitri Zvonkine for useful discussions. We are grateful to the referees for their helpful comments and corrections.

S.~C.~is supported by the European Research Council (ERC) under the European Union's Horizon 2020 research and innovation programme (grant agreement  No.~ERC-2016-STG 716083 ``CombiTop''). E.~G.-F.~was supported by the same grant when this work started. N.~K.~C.~thanks the Max Planck Institute for Mathematics for the excellent working conditions provided. This work is partly a result of the ERC-SyG project, Recursive and Exact New Quantum Theory (ReNewQuantum) which received funding from the European Research Council (ERC) under the European Union's Horizon 2020 research and innovation programme under grant agreement No 810573. A.~G. has been supported by the Institut de Physique Th\'{e}orique Paris (IPhT), CEA, Universit\'{e} de Saclay. We thank the Institut de Math\'ematiques de Jussieu - Paris Rive Gauche for its hospitality. 

\addtocontents{toc}{\protect\setcounter{tocdepth}{2}}

\section{Witten classes and topological recursion}
\label{sec:Witten:TR}

\subsection{Cohomological field theories}

Let us recall some facts about cohomological field theories, CohFTs for short, and the Givental action. The main example treated in this article is that of the Witten $ r $-spin classes.

\begin{definition}
	Let $V$ be a finite dimensional $\Q$-vector space with a non-degenerate symmetric $2$-form $\eta$. A \emph{cohomological field theory} on $(V,\eta)$ consists of a collection $\Omega = (\Omega_{g,n})_{2g-2+n>0}$ of elements
	\begin{equation}
		\Omega_{g,n} \in H^{\bullet}(\overline{\mc{M}}_{g,n}) \otimes (V^{\ast})^{\otimes n}
	\end{equation}
	satisfying the following axioms.
	\begin{enumerate}
		\item[i)] Each $\Omega_{g,n}$ is $S_n$-invariant, where the action of the symmetric group $S_n$ permutes both the marked points of $\overline{\mc{M}}_{g,n}$ and the copies of $(V^{\ast})^{\otimes n}$.

		\item[ii)] Consider the gluing maps
		\begin{equation}
		\begin{aligned}
			q \colon& \overline{\mc{M}}_{g-1,n+2} \longrightarrow \overline{\mc{M}}_{g,n}\,, \\
			r_{h,I} \colon& \overline{\mc{M}}_{h,|I|+1} \times \overline{\mc{M}}_{g-h,|J|+1}  \longrightarrow \overline{\mc{M}}_{g,n}\,,
			\qquad
			I \sqcup J = \set{1,\dots,n}.
		\end{aligned}
		\end{equation}
		Then
		\begin{equation}
		\begin{aligned}
			q^{\ast}\Omega_{g,n}(v_1 \otimes \cdots \otimes v_n)
			& =
			\Omega_{g-1,n+2}(v_1 \otimes \cdots \otimes v_n \otimes \eta^{\dag})\,, \\
			r_{h,I}^{\ast} \Omega_{g,n}(v_1 \otimes \cdots \otimes v_n)
			& =
			(\Omega_{h,|I|+1} \otimes \Omega_{g-h,|J|+1}) \Biggl( \bigotimes_{i \in I} v_i \otimes \eta^{\dag} \otimes \bigotimes_{j \in J} v_j \Biggr),
		\end{aligned}
		\end{equation}
		where $\eta^{\dag} \in V^{\otimes 2}$ is the bivector dual to $\eta$.
	\end{enumerate}
	If the vector space comes with a distinguished element $\bm{1} \in V$, we can also ask for a third axiom:
	\begin{enumerate}
		\item[iii)] Consider the forgetful map
		\begin{equation}
			p \colon \overline{\mc{M}}_{g,n+1} \longrightarrow \overline{\mc{M}}_{g,n}\,.
		\end{equation}
		Then
		\begin{equation}
		\begin{aligned}
			p^{\ast} \Omega_{g,n}(v_1 \otimes \cdots \otimes v_n)
			& =
			\Omega_{g,n}(v_1 \otimes \cdots \otimes v_n \otimes \bm{1})\,, \\
			\Omega_{0,3}(v_1 \otimes v_2 \otimes \bm{1})
			& =
			\eta(v_1,v_2)\,.
		\end{aligned}
		\end{equation}
	\end{enumerate}
	In this case, $\Omega$ is called a \emph{cohomological field theory with a flat unit}.
\end{definition}

The structure of a  cohomological field theory turns $V$ into a commutative algebra. The product structure is known as the \textit{quantum product}, denoted $ \bullet $, and is defined as follows: 
\begin{equation}
	\eta(v_1 \bullet v_2,v_3) = \Omega_{0,3}(v_1 \otimes v_2 \otimes v_3)\,.
\end{equation}
Commutativity follows from (i), associativity from (ii). If the cohomological field theory has a flat unit, the quantum product is unital, with $\bm{1} \in V$ being the identity by (iii).

The degree $0$ part of a CohFT
\begin{equation}
	w_{g,n} \coloneqq \deg_0{\Omega_{g,n}} \in H^{0}(\overline{\mc{M}}_{g,n}) \otimes (V^{\ast})^{\otimes n} \cong (V^{\ast})^{\otimes n}
\end{equation}
is again a CohFT. More precisely, it is a \emph{$2d$ topological field theory} (TFT for short), and is uniquely determined by the values of $\deg_0{\Omega_{0,3}}$ and by the bilinear form $\eta$ (or equivalently, by the associated quantum product). In particular, we will refer to a TFT as being unital and/or semisimple when the associated algebra is so (after extension to $\C$).

In \cite{Giv01}, Givental defined a certain action on Gromov--Witten potentials, and this action was lifted to CohFTs in the work of Teleman \cite{Tel12}. We recall here the basic definitions, and refer to \cite{PPZ15} for more details.

\begin{definition}
	Fix a vector space $V$ with a non-degenerate symmetric bilinear form $\eta$. An \emph{$R$-matrix} is an element $ R(u) = \Id + \sum_{k=1}^\infty R_k u^k \in \Id + u \End (V) \bbraket{u}$ satisfying the symplectic condition:
	\begin{equation}
		R(u) R^{\dag}(-u) = \Id\,.
	\end{equation}
	Here $R^{\dag}$ is the adjoint with respect to $\eta$. The inverse matrix $R^{-1}(u)$ also satisfies the symplectic condition.
\end{definition}

From a CohFT $\Omega$ on $(V,\eta)$ together with an $R$-matrix, one can define a new CohFT on $(V,\eta)$, denoted $R\Omega = (R\Omega_{g,n})_{2g-2+n > 0}$. This defines a left group action on the set of CohFTs on $(V,\eta)$.

There is another type of action on the space of CohFTs known as a translation.

\begin{definition}
	A \emph{translation} is an element $T \in u^2\, V\bbraket{u}$, i.e. a $V$-valued power series vanishing in degree $0$ and $1$:
	\begin{equation}
		T(u) = \sum_{d \ge 1} T_d \, u^{d+1} \, ,
		\qquad
		T_d \in V \, .
	\end{equation}
\end{definition}

From a CohFT $\Omega$ on $(V,\eta)$ together with a translation, one can define a new CohFT on $(V,\eta)$, denoted $T\Omega = (T\Omega_{g,n})_{2g-2+n > 0}$. This defines an abelian group action on the set of CohFTs on $(V,\eta)$.

In general, if we start from a CohFT with unit on $(V,\eta,\bm{1})$, acting by an $R$-matrix or by translation does not preserve flatness. However, there is a specific combination of the two which does achieve this.

\begin{proposition}\label{prop:TL:TR}
	Let $\Omega$ be a CohFT on $(V,\eta,\bm{1})$ with a flat unit. Let $R$ be an $R$-matrix, and consider the $V$-valued power series
	\begin{equation}\label{eqn:TL:TR}
		T_{\textup{L}}(u) \coloneqq u \bigl( R^{-1}(u)\bm{1} - \bm{1} \bigr) \, ,
		\qquad
		T_{\textup{R}}(u) \coloneqq	u \bigl( \bm{1} - R^{-1}(u)\bm{1} \bigr) \, .
	\end{equation}
	Then $T_{\textup{L}}R\Omega$ and $RT_{\textup{R}}\Omega$ coincide, and form a CohFT with a flat unit.
\end{proposition}

\begin{definition}
	Let $\Omega$ be a CohFT on $(V,\eta,\bm{1})$ with a flat unit. Let $R$ be an $R$-matrix. We define the \emph{unit-preserving action} by $R$ as
	\begin{equation}
		R.\Omega \coloneqq T_{\textup{L}}R\Omega = RT_{\textup{R}}\Omega\,.
	\end{equation}
\end{definition}

\subsubsection{Teleman's reconstruction theorem}

One of the main tools in the study of CohFTs is Teleman's reconstruction theorem \cite{Tel12}, which allows one to reconstruct a  homogenous semisimple CohFT from its topological part, an $R$-matrix and a translation. Let us explain the notion of Dubrovin--Frobenius manifolds and homogenous CohFTs (which arise naturally in Gromov--Witten theory) briefly before stating Teleman's reconstruction theorem.

Given a CohFT $\Omega$ on $V$, we can naturally endow (possibly a formal neighbourhood of the origin in) $ V $ with the structure of a Dubrovin--Frobenius manifold by restricting to genus $0$. To be precise, assume that we have a $ d $-dimensional vector space $ V $ with the flat basis $ v_0, \ldots, v_{d-1} $ underlying our CohFT $ \Omega$. Then, we can define the primary genus $0$ potential $ F $ as 
\begin{equation}\label{eqn:defF}
	F (t_0, \ldots, t_{d-1})
	=
	\sum_{\substack{k_0+ \cdots + k_{d-1} = n \\ n\geq 3}}
	\left( \int_{\overline{\mathcal M}_{0,n}} \Omega_{0,n} (v_0^{\otimes k_0} \otimes \cdots \otimes v_{d-1}^{\otimes k_{d-1}} ) \right)
	\prod_{i=0}^{d-1} \frac{t_i^{k_i}}{k_i!} \, ,
\end{equation}
where we view $ t_i $ as the dual coordinate to the basis element $ v_i $. Assuming that the sum in \cref{eqn:defF} converges, $ V $ inherits the structure of a Dubrovin--Frobenius manifold with flat coordinates $ (t_0, \ldots, t_{d-1}) $. All the information of this Dubrovin--Frobenius manifold is encoded in the potential $F$.

We can equip the tangent space at every point $ p $ on the Dubrovin--Frobenius manifold $ V $ with an associative algebra structure given by the \textit{quantum product}
\begin{equation}
	\de_{i} \bullet_p \de_{j}
	=
	\sum_{k,\ell =0}^{d-1} \left.\left( \frac{\partial^3 F}{\partial t_i \partial t_j \partial t_k } \right)\right|_{p} \eta^{k,\ell} \de_{\ell}\,,
\end{equation}
where we introduced the following notation for the vector fields, $\de_i \coloneqq \frac{\partial}{\partial t_i} \in H^0(V, TV) $, which is associated to the basis vector $ v_i $.

We can often equip the Dubrovin--Frobenius manifold $V$ with an additional grading using the notion of an Euler field. An \textit{Euler field} on a Dubrovin--Frobenius manifold $ V $ with flat coordinates $ (t_0, \ldots, t_{d-1})$, is an affine vector field $ E $ satisfying the following conditions.
\begin{itemize}
	\item The vector field $ E $ has the form $E = \sum_i (\alpha_i t_i  + \beta_i ) \de_i\,$.
	
	\item
	The metric $ \eta $ and the quantum product $ \bullet $ are eigenfunctions of the Lie derivative $\mc{L}_E $ with weights $ 2- \delta $ and $1$ respectively, where $ \delta $ is a rational number called the \textit{conformal dimension}.
\end{itemize}

The Euler field $ E $ on the Dubrovin--Frobenius manifold $ V $ can be used to define an action of $ E $ on the CohFT as follows:
\begin{equation}\label{eqn:Eactiongen}
	\begin{split}
		(E.\Omega)_{g,n} \left(\de_{a_1} \otimes \cdots \otimes \de_{a_n}\right)
		& \coloneqq
		\left( \deg + \sum_{l=1}^n \alpha_{a_l} \right) \Omega_{g,n} \left(\de_{a_1} \otimes \cdots \otimes \de_{a_n}\right) \\
		& \qquad +
		p_* \Omega_{g,n+1} \left(\de_{a_1} \otimes  \cdots \otimes \de_{a_n} \otimes \sum_i \beta_i \de_i \right).
	\end{split}
\end{equation}

\begin{definition}
	We say that the CohFT $ \Omega $ is \textit{homogeneous} if there exists an Euler field $E$ such that
	\begin{equation}
		(E.\Omega)_{g,n} = \bigl( (g-1)\delta + n \bigr) \Omega_{g,n}\,.
	\end{equation}
\end{definition}

Finally, we are ready to state Teleman's reconstruction theorem.

\begin{theorem}[{Teleman reconstruction \cite{Tel12}}]\label{thm:Teleman}
	Let $\Omega_{0,n}$ be a genus zero homogeneous semisimple CohFT.
	\begin{itemize}
		\item There exists a unique homogeneous CohFT $\Omega_{g,n}$ that extends $\Omega_{0,n}$ in higher genus.
		
		\item This extended CohFT $\Omega_{g,n}$ is obtained by first applying a translation action $T(u)$ on the topological field theory $ w_{g,n} $ (determined by $\Omega_{0,3}$), and then an $R$-matrix action $R(u)$ to the resulting CohFT.
		
		\item The $R$-matrix and the translation are uniquely specified by the associated Dubrovin--Frobenius manifold and the Euler field.
	\end{itemize}
	If the CohFT $\Omega_{0,n}$ has a flat unit, the translation and the $R$-matrix action can be combined into a unit preserving $R$-matrix action, and thus the extension $\Omega_{g,n}$ is a CohFT with a flat unit.
\end{theorem}

The $ R $-matrix and the translation can be explicitly computed, and we refer the reader to \cite{Tel12,PPZ15,CGG22} for a detailed explanation of this algorithm\footnote{Teleman's algorithm requires one to work over $ \mathbb C $ to compute the $ R $-matrix and the translation $ T $. However, if the original CohFT $ \Omega $ is defined over $ \mathbb Q $, so is the reconstructed CohFT.}.

\subsection{Witten classes}

In this article, we are primarily interested in a certain CohFT known as the Witten $ r $-spin CohFT. For $r \ge 2$, let $V = \Q\braket{v_0,\dots,v_{r-2}}$ with pairing $\eta(v_a,v_b) = \delta_{a+b,r-2}$ and unit $\bm{1} = v_0$. The \textit{Witten $r$-spin theory} is an $(r-1)$-dimensional CohFT with a flat unit
\begin{equation}
	W_{g,n}^{r} \colon V^{\otimes n} \longrightarrow H^{\bullet}(\overline{\mc{M}}_{g,n})\,.
\end{equation}
For $a_1,\ldots,a_n\in\{0,\ldots,r-2\}$, the  Witten $r$-spin class $W_{g,n}^r(v_{a_1} \otimes \cdots \otimes v_{a_n})$ has pure complex degree
\begin{equation}
	D_{g;a}^{r}
	=
	\frac{(r-2)(g-1) + |a|}{r}\,,
\end{equation} where $ |a| \coloneqq \sum_{i=1}^{n} a_i $. If $D_{g;a}^{r}$ is not an integer, the corresponding Witten class vanishes. In genus $0$, the construction was first carried out by Witten \cite{Wit93} using $r$-spin structures, and we briefly recall it here.

Let $\overline{\mc{M}}_{0,n}^{r}(a_1,\dots,a_n)$ be the moduli space parametrising $r$-th roots of the form
\begin{equation}
	L^{\otimes r} \cong \omega_{C}\left( - \sum_{i=1}^n a_i p_i \right) ,
\end{equation}
where $[C,p_1,\dots,p_n] \in \overline{\mc{M}}_{0,n}$ and $\omega_C$ denotes the canonical line bundle on $C$. There is a forgetful map $\pi \colon \overline{\mc{M}}_{0,n}^{r}(a_1,\dots,a_n) \to \overline{\mc{M}}_{0,n}$ that forgets the $r$-th root. Consider the vector bundle $\mc{V}_{0,n}^{r}(a_1,\dots,a_n) \to \overline{\mc{M}}_{0,n}^{r}(a_1,\dots,a_n)$ whose fibre over $[C,p_1,\dots,p_n,L]$ is $H^1(C,L)^{\ast}$. Then the Witten $r$-spin class is defined as the push-forward of the (normalised) top Chern class of $\mc{V}_{0,n}^{r}(a_1,\dots,a_n)$:
\begin{equation}
	\frac{1}{r} \, W_{0,n}^{r}(v_{a_1} \otimes \cdots \otimes v_{a_n})
	=
	\pi_{\ast} c_{\textup{top}}\left( \mc{V}_{0,n}^{r}(a_1,\dots,a_n) \right) .
\end{equation}
The definition of the Witten class in higher genera is much more complicated, since the fibres $H^1(C,L)^{\ast}$ no longer glue together to form a vector bundle. The construction was first carried out by Polishchuk--Vaintrob \cite{PV00}, and was later simplified by Chiodo \cite{Chi06}.

The Witten $ r $-spin class is homogeneous with respect to the Euler field
\begin{equation}
	E = \sum_{a=0}^{r-2} \left( 1-\frac{a}{r} \right) t_a \de_a
\end{equation}
and of conformal dimension $\delta = \frac{r-2}{2}$. However, The Witten $ r $-spin class is not semisimple, so Teleman's reconstruction theorem does not apply. A workaround is to consider a shift of the Witten class along a vector $\gamma \in V$.

\begin{definition}
	The \textit{shifted Witten class} $ \mathsf W^\gamma_{g,n} $, shifted along the vector $ \gamma \in V $, is defined as 
	\begin{equation}
		\mathsf W^{r,\gamma}_{g,n}(v_{a_1} \otimes \cdots \otimes v_{a_n} )  := \sum_{m \ge 0} \frac{1}{m!} \, p_{m,\ast} W_{g,n+m}^{r} \bigl( v_{a_1} \otimes \cdots \otimes v_{a_n} \otimes \gamma^{\otimes m} \bigr) \, ,
	\end{equation} where $ p_m \colon \overline{\mc{M}}_{g,n+m}\rightarrow \overline{\mc{M}}_{g,n}$ denotes the forgetful map that forgets the last $ m $ marked points. 
\end{definition} The sum appearing in the definition of the shifted Witten class is easily seen to be finite due to degree considerations. Moreover, the term of highest degree in  $ \mathsf W^\gamma_{g,n} $ is the Witten class $ W^r_{g,n} $. 

For any shift $ \gamma $, the classes $ \mathsf W^\gamma_{g,n} $ form a CohFT with a flat unit. Moreover, for a generic shift, the shifted Witten class is semisimple and one can apply Teleman's reconstruction theorem to the shifted class. For two specific $1$-parameter families of shifts, namely $\gamma = \epsilon r v_{r-2}$ and $\gamma = \epsilon r v_{1}$, Pandharipande--Pixton--Zvonkine \cite{PPZ19} computed all the ingredients required for Teleman's reconstruction theorem as we now describe.

\subsubsection{The \texorpdfstring{$v_{r-2}$ shift}{shift along (0,...,0,r)}}

Consider the vector $ \gamma = \epsilon r v_{r-2} $, where $ \epsilon \in \mathbb C $ is a formal parameter. For notational convenience, we denote the $ ( \epsilon r v_{r-2}) $-shifted Witten class as
\begin{equation}
	\widehat{W}_{g,n}^{r,\epsilon}(v_{a_1} \otimes \cdots \otimes v_{a_n})
	:= \mathsf W^{r,\epsilon r v_{r-2}}_{g,n}(v_{a_1} \otimes \cdots \otimes v_{a_n})\,.
\end{equation}
By  applying  Teleman's reconstruction theorem, $ 	\widehat{W}_{g,n}^{r,\epsilon} $ can be expressed using a unit-preserving $R$-matrix action as
\begin{equation}
	\widehat{W}_{g,n}^{r,\epsilon} = R.\widehat{w}_{g,n}^{r}\,,
\end{equation}
where the TFT and the $R$-matrix are given by (cf.~\cite[section~2]{PPZ19})
\begin{equation}
	\widehat{w}_{g,n}^{r}(v_{a_1} \otimes \cdots \otimes v_{a_n})
	=
	\epsilon^{\frac{(r-2)(g-1)+|a|}{2}}\left( \frac{r}{2} \right)^{g-1}
	\sum_{k=1}^{r-1} (-1)^{(k-1)(g-1)}
		\frac{\prod_{i=1}^n \sin\bigl(\frac{(a_i+1)k\pi}{r}\bigr)}{\sin^{2g-2+n}(\frac{k\pi}{r})}\,,
\end{equation}
\begin{equation}
	R^{-1}(u)_a^a = \ms{B}_{r,a}^{\textup{even}} ( \epsilon^{-r/2} u ) \, ,
	\qquad
	R^{-1}(u)_a^{r-2-a} = \ms{B}_{r,a}^{\textup{odd}} ( \epsilon^{-r/2} u ) \, ,
\end{equation}
with the $R$-matrix elements being $0$ elsewhere. If $r$ is even, the coefficient at the intersection of the diagonals is set to be $1$. Here $\ms{B}_{r,a}(u)$ are the hypergeometric series ($a = 0,\dots,r-2$)
\begin{equation}\label{eqn:B:series}
	\ms{B}_{r,a}(u)
	=
	\sum_{m \ge 0}\left(
		\prod_{k=1}^m \frac{\bigl( (2k-1)r - 2(a+1) \bigr)\bigl( (2k-1)r + 2(a+1) \bigr)}{k}
	\right)
	\left(
		-\frac{u}{16r^2}
	\right)^{m},
\end{equation}
and we denote by $\ms{B}_{r,a}^{\textup{even}}$ (resp.~$\ms{B}_{r,a}^{\textup{odd}}$) the even (resp.~odd) summands of $\ms{B}_{r,a}$.

\subsubsection{The \texorpdfstring{$v_{1}$ shift}{shift along (0,r,...,0)}}

Consider the vector $ \gamma = \epsilon r v_{1} $, where $ \epsilon \in \mathbb C $ is a formal parameter. We denote the $  ( \epsilon r v_{1}) $-shifted Witten class as 
\begin{equation}
	\widetilde{W}_{g,n}^{r,\epsilon}(v_{a_1} \otimes \cdots \otimes v_{a_n}) := \mathsf W^{r,\epsilon r v_{1}}_{g,n}(v_{a_1} \otimes \cdots \otimes v_{a_n}) \, .
\end{equation}
Again, $ \widetilde{W}_{g,n}  $ is semisimple and is given by a unit-preserving action as
\begin{equation}
	\widetilde{W}_{g,n}^{r,\epsilon} = R.\widetilde{w}_{g,n}^{r}\,,
\end{equation}
where the TFT and the $R$-matrix elements are given by (cf.~\cite[section~4]{PPZ19})
\begin{equation}
	\widetilde{w}_{g,n}^{r}(v_{a_1} \otimes \cdots \otimes v_{a_n})
	=
	\epsilon^{\frac{(g-1)(r-2) + |a|}{r-1} }(r-1)^g \cdot \delta \, ,
\end{equation}
\begin{equation}
	(R^{-1}_m)_a^b = \left( \frac{1}{r(r-1) \epsilon^{r/(r-1)}} \right)^m P_m(r,a) \,,
	\qquad
	\text{if } b+m \equiv a \pmod{r-1}\,,
\end{equation}
and are  $0$ otherwise. Here $\delta$ equals $1$ if $g-1 \equiv |a| \pmod{r-1}$ and $0$ otherwise, while the coefficients $P_m(r,a)$ are computed recursively as
\begin{equation}\label{eqn:PPZ:polys}
	\begin{cases}
		P_m(r,a) - P_m(r,a-1) = r \left( m - \frac{1}{2} - \frac{a}{r} \right) P_{m-1}(r,a-1)\,,
		\qquad
		\text{for } a = 1,\dots,r-2\,, \\
		P_m(r,0) = P_m(r,r-1)\,,
	\end{cases}
\end{equation}
with initial condition $P_0 = 1$.

\subsection{Topological recursion}
\label{subsec:TR}

Topological recursion, TR for short, is a universal procedure that associates a collection of symmetric multidifferentials to a spectral curve, a curve with some extra data \cite{EO07}. What makes TR especially useful is its applications to enumerative geometry: many counting problems are solved by TR, in the sense that the sought numbers are coefficients of the multidifferentials when expanded in a specific base of differentials.

\begin{definition}\label{defn:spectral:curve}
	A \emph{spectral curve} $\mc{S} = (\Sigma,x,y,\omega_{0,2})$ consists of
	\begin{itemize}
		\item a Riemann surface $\Sigma$ (not necessarily compact or connected);
		\item a meromorphic function $x \colon \Sigma \to \C$ such that the differential $dx$  has finitely many zeros (called ramification points) $\alpha_1,\dots,\alpha_r$ that are simple; 
		\item a meromorphic function $y \colon \Sigma \to \C$ such that $dy$ does not vanish at the zeros of $dx$;
		\item a symmetric bidifferential $\omega_{0,2}$ on $\Sigma \times \Sigma$, with a double pole on the diagonal with biresidue $1$, and no other poles.
	\end{itemize}
	The \emph{topological recursion} produces symmetric multidifferentials (also called \emph{correlators}) $\omega_{g,n}$ on $\Sigma^n$, defined recursively on $2g-2+n > 0$ as
	\begin{equation}\label{eqn:TR}
	\begin{split}
		\omega_{g,n}(z_1,\dots,z_n)
		&\coloneqq
		\sum_{i=1}^r \Res_{z = \alpha_i} K_i(z_1,z) \bigg(
			\omega_{g-1,n+1}(z,\sigma_i(z),z_2,\dots,z_n) \\
		&\qquad\quad +
			\sum_{\substack{g_1+g_2 = g \\ N_1 \sqcup N_2 = \{2,\dots,n\}}}^{\text{no $(0,1)$}}
				\omega_{g_1,1+|N_1|}(z,z_{N_1}) \,
				\omega_{g_2,1+|N_2|}(\sigma_i(z),z_{N_2})
		\bigg)\,,
	\end{split}
	\end{equation}
	where $K_i$, called the topological recursion kernels\footnote{In older references, such as \cite{EO07}, the recursion kernel has an overall minus sign, but the convention of this paper is now considered standard. This minus sign would result in a $ (-1)^n $ rescaling of $ \omega_{g,n} $.}, are locally defined in a neighbourhood $U_i$ of $\alpha_i$ as
	\begin{equation}\label{eqn:TR:kernel}
		K_i(z_1,z) \coloneqq \frac{\frac{1}{2} \int_{w = \sigma_i(z)}^z \omega_{0,2}(z_1,w)}{\bigl( y(z) - y(\sigma_i(z)) \bigr) d x(z)}\,,
	\end{equation}
	and $\sigma_i \colon U_i \to U_i$ is the Galois involution near the ramification point $\alpha_i \in U_i$. The ``no $(0,1)$" in equation \eqref{eqn:TR} means that one should omit the cases where $ (g_1, 1+|N_1|)  = (0,1)$ or  $ (g_2, 1+|N_2|)  = (0,1)$. It can be shown that $\omega_{g,n}$ is a symmetric meromorphic multidifferential on $\Sigma^n$ with poles only at the ramification points.
\end{definition}

If we relax the assumptions on the spectral curve by allowing the ramification points of $x \colon \Sigma \to \C$ to be not necessarily simple, we have to apply a generalised version of the topological recursion known as the Bouchard--Eynard topological recursion \cite{BE13, BHLMR14}. The Bouchard--Eynard topological recursion was proved to be well-defined (i.e., it produces symmetric correlators) if and only if an \textit{admissibility condition} is satisfied \cite{BBCCN23}. 

For the sake of  notation and since it is sufficient for this paper, we present the Bouchard--Eynard topological recursion (and state the admissibility condition) only in the case where $ x $ has a single ramification point where $ dx $ has a zero of order $ r-1 $. Around this ramification point, we assume that $ y $ behaves as
\begin{equation}
	y(z) \sim z^{s-r} + O(z^{s-r+1})\,,
\end{equation}
where $1 \leq s \leq r+1 $. Then the admissibility condition states that $ s \equiv \pm 1 \pmod{r} $. 

\begin{definition}\label{def:localTR}
For $r\in\mathbb{Z}_{> 0}$, set $[r]\coloneqq \{1,\dots,r\}$.	Consider an admissible spectral curve such that $d x$ has a single zero of degree $r-1$ at $ z = \alpha $, and let $ \sigma $ denote a  generator of the group of local deck transformations around $ z = \alpha $. Let $o\in\Sigma$ be an arbitrary base point on the spectral curve. The \emph{Bouchard--Eynard topological recursion} produces symmetric multidifferentials $\omega_{g,n}$ on $\Sigma^n$ recursively on $2g-2+n>0$ as:
	\begin{equation}\label{eqn:local:TR}
	\begin{split}
		\omega_{g,n}(z_1,\dots,z_n)
		\coloneqq
		\frac{1}{2\pi\iu} \oint_{z \in C} \sum_{\substack{I\subset [r-1] \\ I\neq \varnothing}}
		(-1)^{|I|+1}
		\frac{\int_{w=o}^{z} \omega_{0,2}(z_1,w)}{\prod_{i\in I} \bigl( \omega_{0,1}(z)-\omega_{0,1}(\sigma^i(z)) \bigr)} & \\
		\times
		\sum_{\substack{J\vdash I\cup \{r\} \\ \sqcup_{i=1}^{\ell(J)} N_i =\{2,\dots,n\} \\ \sum_{i=1}^{\ell(J)} g_i = g+\ell(J)-|I|-1 }}^{\textup{no } (0,1)}
			\prod_{i=1}^{\ell(J)} \omega_{g_i,|J_i|+|N_i|}(\sigma_{J_i}(z),z_{N_i}) \, ,
	\end{split}
	\end{equation}
	where $C$ is a contour around the ramification point $\alpha$, $(\sigma^{1}(z),\dots, \sigma^{r-1}(z),\sigma^{r}(z)=z)$ are the images of $z$ under the local deck transformation group around $\alpha$, and the notation $\sigma_{J}(z)$ stands for $(\sigma^j(z))_{j \in J}$ and $\sigma$. Finally, in the second sum the cases where $(g_i, |J_i|+|N_i|)=(0,1)$ are forbidden. 
\end{definition}

As mentioned previously, the fact that \cref{eqn:local:TR} produces well-defined symmetric correlators was proved in \cite[theorem~E]{BBCCN23}.

\begin{remark}\label{rem:ltog}
	When we have a compact spectral curve, \cite{BE13} defined a global version of the topological recursion, where the main difference from the local version~\eqref{eqn:local:TR} is that the set of local deck transformations $( \sigma^1(z), \ldots, \sigma^{r-1}(z), \sigma^r(z) )$ at the ramification point $z=a$ is replaced by the set of global sheet involutions (which does not depend on the ramification index at each ramification point, but rather on the degree of the branched covering $ x $). In the case of simple ramification points, they prove that it coincides with the standard definition~\eqref{eqn:TR} of the topological recursion, and that in the setting of \cref{def:localTR} it coincides with \cref{eqn:local:TR}.
\end{remark}

\subsubsection{TR-CohFT correspondence}
The topological recursion correlators are closely related to intersection numbers on the moduli space of curves. First, \cite{Eyn14+} proved that the topological recursion  computes certain intersection numbers on $ \overline{\mc{M}}_{g,n} $. Afterwards, \cite{DOSS14} proved that the descendant integrals of a semisimple CohFT are computed by topological recursion on a local spectral curve. Finally, \cite{DNOPS18} established in the converse direction that topological recursion on a \textit{compact} spectral curve computes the descendant integrals of an associated semisimple CohFT (by refining results of \cite{Eyn14+} and connecting it to the Givental formalism). We refer to this collection of results as the ``TR-CohFT correspondence''.

We explain how to compute the   CohFT associated to a compact spectral curve  $ \mc{S} = (\Sigma, x,y,\omega_{0,2}) $ following \cite{DNOPS18}. We will work with CohFTs over $\C$.  Let us fix a global constant $C \in \C^{\times}$. Choose a local coordinate $\zeta_i$ on the neighbourhood  $U_i$ of the ramification point $ \alpha_i  $ such that $\zeta_i(\alpha_i) = 0$ and $x - x(\alpha_i) = - \frac{\zeta_i^2}{2}$. We define the auxiliary functions $\xi^i \colon \Sigma \to \C$ and the associated differentials $d\xi^{k,i}$ as
\begin{equation}
	\xi^i(z)
	\coloneqq
	\int^z \left.\frac{\omega_{0,2}(\zeta_{i},\cdot)}{d\zeta_{i}}\right|_{\zeta_{i} = 0}\,,
	\qquad
	d\xi^{k,i}(z)
	\coloneqq
	d\biggl( \frac{d^k}{dx(z)^k} \xi^{i}(z) \biggr)\,.
\end{equation}
We also set $\Delta^i \coloneqq \frac{dy}{d\zeta_i}(0)$ and $h^i \coloneqq C \Delta^i$. We define a unital, semisimple TFT on $V \coloneqq \C\braket{e_1,\dots,e_r}$ by setting
\begin{equation}
	\eta(e_i,e_j) \coloneqq \delta_{i,j}\,,
	\quad
	\bm{1}
	\coloneqq
	\sum_{i=1}^r h^i e_i\,,
	\quad
	w_{g,n}(e_{i_1} \otimes \cdots \otimes e_{i_n})
	\coloneqq
	\delta_{i_1,\ldots,i_n} (h^i)^{-2g+2-n}\,.
\end{equation}
Define the $R$-matrix $R \in \End(V)\bbraket{u}$ and the translation $T \in u^2 V\bbraket{u}$ by setting
\begin{align}
	\label{eqn:R:matrix:TR}
	R^{-1}(u)^j_i
	& \coloneqq
	- \sqrt{\frac{u}{2\pi}} \int_{\gamma_j}
 		d\xi^i \,
 		e^{\frac{1}{u}(x - x(\alpha_j))}\,, \\
	\label{eqn:translation:TR}
	T^i(u)
 	& \coloneqq
 	u h^i - C \sqrt{\frac{u}{2\pi}} \int_{\gamma_i} dy \, e^{\frac{1}{u}(x - x(\alpha_i))}\,.
\end{align}
Here $\gamma_i$ is the Lefschetz thimble passing through the critical point $\alpha_i$, and the equations are intended as equalities between formal power series in $u$, where on the right-hand side we take the formal asymptotic expansion as $u \to 0$ (cf.~\cref{app:exp:integrals}). Through the Givental action, we can then define a cohomological field theory
\begin{equation}
	\Omega_{g,n} \coloneqq RTw_{g,n} \in H^{\bullet}(\overline{\mc{M}}_{g,n}) \otimes (V^{\ast})^{\otimes n}
\end{equation}
from the data $(w,R,T)$. One direction of the  TR-CohFT correspondence is given by the following theorem (for the other direction, see \cite{DOSS14}). 

\begin{theorem}[{\cite{DNOPS18}}]\label{thm:Eyn:DOSS}
	Suppose we have a compact spectral curve with simple ramification points $(\Sigma, x,y,B)$. Then its topological recursion correlators are given by
	\begin{multline}
		\omega_{g,n}(z_1,\dots,z_n) = \\
		=
		C^{2g-2+n}
		\sum_{i_1,\dots,i_n = 1}^r \int_{\overline{\mc{M}}_{g,n}}
			\Omega_{g,n}(e_{i_1} \otimes\cdots\otimes e_{i_n}) \prod_{j=1}^n \sum_{k_j \ge 0}
				\psi_j^{k_j} d\xi^{k_j,i_j}(z_j) \, .
	\end{multline}
\end{theorem}

\section{Generalised Airy functions}
\label{sec:gen:Airy}

We are interested in two possible generalisations of the Airy differential equation $\ddot{u}(t) = tu(t)$, for which we will present a basis of asymptotic solutions constructed from integral representations. This will generalise the asymptotic formula for the Airy function, valid as $t \to \infty$ and $|\arg(t)| < \pi$:
\begin{equation}
	\Ai(t)
	=
	\frac{1}{2\pi\iu} \int_{C} e^{\frac{w^3}{3} - tw} dw
	\sim
	\frac{e^{-\frac{2t^{3/2}}{3}}}{2\sqrt{\pi}} t^{-1/4} \, \sum_{k \ge 0}
 		\frac{(6k)!}{(2k)! (3k)!}
 		\left( -\frac{1}{576 t^{3/2}} \right)^k.
\end{equation}
Here $C$ is the path starting at $e^{-\frac{\pi}{3}} \infty$ and ending at $e^{\frac{\pi}{3}} \infty$. We refer to \cref{app:exp:integrals} for some facts about exponential integrals, Lefschetz thimbles and asymptotic expansions.

\subsection{The \texorpdfstring{$\widehat{\Ai}$}{Airy--Lucas} functions}

For fixed integers $r \geq 3$ and $0 \leq a \leq r-2$, let us consider the second order differential equation
\begin{equation}\label{eqn:LA}
	\ddot{u}(t) = t^{r-2} \, u(t) + \frac{a}{t} \, \dot{u}(t) \,.
\end{equation}
For fixed $ a $, we can define $ r-1 $ solutions that we call \emph{Airy--Lucas functions}, via contour integrals as
\begin{equation}\label{eqn:LAiry}
	\hAi_{r,a}^{j}(t)
	\coloneqq
	\frac{1}{2\pi\sqrt{(-1)^{j}}} \int_{\hat{C}_j} U_{a+1}(w,t) e^{\frac{1}{r} V_r(w,t)} dw\,,
	\qquad\quad
	j = 1,\dots,r-1\,,
\end{equation}
where $\sqrt{(-1)^j}$ is equal to the imaginary unit $\iu = \sqrt{-1}$ if $j$ is odd and is equal to $1$ if $j$ is even. The $U_n$ and $V_n$ are Lucas polynomial sequences (cf.~\cref{app:Lucas}) defined recursively by
\begin{equation}\label{eqn:rec:Lucas}
	\begin{split}
	& \begin{cases}
		U_0(w,t) = 0\,, \\
		U_1(w,t) = 1\,, \\
		U_n(w,t) = w U_{n-1}(w,t) - t U_{n-2}(w,t)\,,
	\end{cases}	
	\\
	& \begin{cases}
		V_0(w,t) = 2\,, \\
		V_1(w,t) = w\,, \\
		V_n(w,t) = w V_{n-1}(w,t) - t V_{n-2}(w,t)\,,
	\end{cases}
\end{split}
\end{equation}
and the contour $\hat{C}_j$ is the Lefschetz thimble passing through the critical point $w = 2\sqrt{t} \cos(\frac{j\pi}{r})$. Notice that the Airy--Lucas functions are defined for generic phases of $t$ (those that avoid Stokes rays). In this case, each Lefschetz thimble passes through a single critical point. As functions of $t$, the Lefschetz thimbles $\hat{C}_j$ might jump when $t$ crosses a Stokes ray; however, since we will be interested in asymptotic solutions only, we will assume $t$ to be away from Stokes rays and work in the good sectors.

In \cref{prop:LA:ODE:solution} we will show that the Airy--Lucas functions $\hAi_{r,a}^{j}$ are indeed solutions to the ODE \eqref{eqn:LA}, and linearly dependent for $j$s with the same parity. We can compute their asymptotics as $t \to \infty$ via the steepest descent method.

\begin{center}
\begin{tabular}{ c | l l }
	\toprule
	$n$ & $U_n(w,t)$ & $V_n(w,t)$
	\\
	\midrule
	$0$ & $0$ & $2$
	\\
	$1$ & $1$ & $w$
	\\
	$2$ & $w$ & ${w}^{2}-2t$
	\\
	$3$ & ${w}^{2}-t$ & ${w}^{3}-3wt$
	\\
	$4$ & ${w}^{3}-2wt$ & ${w}^{4}-4{w}^{2}t+2{t}^{2}$
	\\
	$5$ & ${w}^{4}-3{w}^{2}t+{t}^{2}$ & ${w}^{5}-5{w}^{3}t+5w{t}^{2}$
	\\
	$6$ & ${w}^{5}-4{w}^{3}t+3w{t}^{2}$ & ${w}^{6}-6{w}^{4}t+9{w}^{2}{t}^{2}-2{t}^{3}$
	\\
\bottomrule
\end{tabular}
\end{center}

\begin{lemma}\label{lem:LAexp}
	The asymptotic behaviour as $t \to \infty$ of the Airy--Lucas functions is given by
	\begin{equation}
		\hAi_{r,a}^{j}(t)
		\sim
		\frac{e^{ (-1)^j \frac{2}{r} t^{r/2}}}{\sqrt{\pi r}}
		\sin{ \left( \frac{(a+1)j\pi}{r} \right) }
		t^{(2a+2-r)/4}
		\left(	1	+	O\bigl( t^{-r/2} \bigr)	\right).
	\end{equation}
\end{lemma}

\begin{proof}
	From the recursion relations \eqref{eqn:rec:Lucas}, it can be easily seen that Lucas sequences are related to Chebyshev polynomials as
	\begin{equation*}
		U_n(w,t) = t^{\frac{n-1}{2}} \, \mc{U}_{n-1} \bigl( \tfrac{w}{2\sqrt{t}} \bigr) \,,
		\qquad
		V_n(w,t) = 2 t^{\frac{n}{2}} \, \mc{T}_{n} \bigl( \tfrac{w}{2\sqrt{t}} \bigr) \,.
	\end{equation*}
	Here $\mc{T}_m$ and $\mc{U}_m$ are the Chebyshev polynomials of the first and second kind respectively. By applying the change of variables $z = \frac{w}{2 \sqrt{t}}$, we can write
	\begin{equation}\label{eqn:UtoI}
		\hAi_{r,a}^{j}(t)
		=
		\frac{t^{\frac{a+1}{2}}}{\pi \sqrt{(-1)^j}}
		\int_{\hat{\Gamma}_j} \mc{U}_{a}(z) e^{\frac{2}{r} t^\frac{r}{2} \mc{T}_r(z)} dz \, .
	\end{equation}
	Here $\hat{\Gamma}_j$ is the Lefschetz thimble in the $z$-plane that passes through the critical point $z = \cos( \frac{j\pi}{r})$. We can now apply the steepest descent method to the function
	\begin{equation*}
		I(s) = \int_{\Gamma} g(z) e^{s h(z)} dz\,,
	\end{equation*}
	with $s = t^{r/2}$, $g(z) = \mc{U}_{a}(z)$, and $h(z) = \frac{2}{r} \mc{T}_r(z)$. This directly gives the following asymptotic expansion:
	\begin{equation*}
		I(s)
		\sim
		g(z_j)
		\sqrt{\frac{2\pi}{s |h''(z_j)|}}
		e^{s h(z_j)} e^{\iu \vartheta(z_j)}
		\left(
			1 + O\bigl( s^{-1} \bigr)
		\right),
	\end{equation*}
	where the phase $\vartheta$ is given by $\vartheta(z_j) = \frac{\pi - \arg(h''(z_j))}{2}$. By using the properties of the Chebyshev polynomials, we see that
	\[
		g(z_j)
		=
		\frac{\sin((a+1)j\pi/r)}{\sin(j\pi/r)}\,,
		\qquad
		h(z_j) = (-1)^j\frac{2}{r}\,,
		\qquad
		h''(z_j) = (-1)^{j+1} \frac{2r}{\sin^2(j\pi/r)}\,.
	\]
	Thus, we finally get
	\begin{equation*}
		I(t^{r/2})
		\sim
		\sin{ \left( \frac{(a+1)j\pi}{r} \right)}
		\sqrt{\frac{\pi}{t^{r/2} r}}
		e^{ (-1)^j \frac{2}{r} t^{r/2}}
		\sqrt{(-1)^j}
		\left(	1	+	O\bigl( t^{-r/2} \bigr)	\right).
	\end{equation*}
	Plugging this asymptotic expansion into \cref{eqn:UtoI} gives the statement of the \lcnamecref{lem:LAexp}.
\end{proof}

We can get the full asymptotic expansion by means of the differential equation, and the coefficients of the asymptotic expansion are expressed in terms of the hypergeometric series $\ms{B}_{r,a}$ defined by \cref{eqn:B:series}.

\begin{proposition}\label{prop:LAasymp}
	The asymptotic expansion as $t \to \infty$ of the Airy--Lucas function is given by
	\begin{equation}
		\hAi_{r,a}^{j}(t)
		\sim
		\frac{e^{ (-1)^j \frac{2}{r} t^{r/2}}}{\sqrt{\pi r}} \,
		\sin{ \left( \frac{(a+1)j\pi}{r} \right) } \,
		t^{(2a+2-r)/4} \,
		\ms{B}_{r,a} \left( (-1)^{j+1} r t^{-r/2} \right) .
	\end{equation}
\end{proposition}

\begin{proof}
	By \cref{lem:LAexp}, we know that the asymptotic expansion of the Airy--Lucas function has the following form:
	\[
		\hAi_{r,a}^{j}(t)
		\sim 
		\frac{e^{ (-1)^j \frac{2}{r} t^{r/2}}}{\sqrt{\pi r}}
		\sin{ \left( \frac{(a+1)j\pi}{r} \right) }
		t^{(2a+2-r)/4}
		\sum_{m \geq 0} b_m(r,a,j) t^{-(rm)/2}, 
	\]
	where $b_0(r,a,j) = 1$. Then, we can plug this ansatz into the differential \cref{eqn:LA} to obtain a recursion for the coefficients $b_m(r,a,j)$.
	For the sake of notational simplicity, let us define the series
	\[
		v(t) = \sum_{m \geq 0} b_m(r,a,j) t^{(a+1-rm)/2 -r/4} \, .
	\]
	After some elementary algebra, the differential \cref{eqn:LA} yields
	\begin{equation*}
		\ddot{v}(t)
		+
		\left( (-1)^j 2 t^{r/2} - a  \right) t^{-1} \, \dot{v}(t)
		+
		(-1)^j \left( (\tfrac{r}{2}-a-1) t^{r/2} \right) t^{-2}  \, v(t)
		=
		0 \, .
	\end{equation*}
	By picking the power of $t^{(a+1-rm)/2 -r/4 - 2} $, we get
	\begin{equation*}
		b_m(r,a,j)\left[ \left( \tfrac{a+1-rm}{2} - \tfrac{r}{4} \right)
		\left( \tfrac{-a-1-rm}{2} - \tfrac{r}{4} \right) \right]
		=
		(-1)^j r(m+1) \, b_{m+1}(r,a,j) \, .
	\end{equation*}
	This in turns implies that 
	\[
		b_{m+1}(r,a,j)
		=
		(-1)^{j}
		\frac{\big((2m+1)r - (2a+2) \big) \big( (2m+1)r + (2a+2) \big) }{16r(m+1)}
		b_m(r,a,j) \,.
	\]
	We can solve this recursion easily to get
	\[
		b_m(r,a,j) = \frac{(-1)^{jm}}{(16r)^m} \prod_{k=1}^m \frac{\bigl( (2k-1)r - 2(a+1) \bigr)\bigl( (2k-1)r + 2(a+1) \bigr)}{k} \,.
	\]
	Comparing these coefficients with the definition \eqref{eqn:B:series} of $\ms{B}_{r,a}$, we get the thesis.
\end{proof}

\subsection{The \texorpdfstring{$\tAi$}{hyper-Airy} functions}

For $r \ge 3$, let us consider the $r$-th order ordinary differential equation
\begin{equation}\label{eqn:hyperAiryeq}
	u^{(r-1)}(t) = (-1)^{r-1} \, t \, u(t) \, ,
\end{equation}
which is sometimes called higher order Airy equation. We can define $r-1$ (independent) solutions via contour integrals that we call \emph{hyper-Airy functions}, as
\begin{equation}\label{eqn:hyperAiry}
	\tAi_{r,k}(t) \coloneqq \frac{\theta^{k\frac{r-2}{2}}}{2\pi\iu} \int_{\tilde{C}_k} e^{\frac{w^{r}}{r} - t w} dw\,,
	\qquad\quad
	k = 0,\dots,r-2\,.
\end{equation}
where $\theta = e^{\frac{2\pi\iu}{r-1}}$ and $\tilde{C}_k$ is the Lefschetz thimble passing through the critical point $w = t^{1/(r-1)} \theta^k$. As in the previous section, we will be interested in the asymptotic expansion of the above functions, and work in the sectors of the $t$-plane where the Lefschetz thimbles are well-defined.

The hyper-Airy functions indeed satisfy \cref{eqn:hyperAiryeq}:
\begin{equation*}
\begin{split}
	\left( \frac{d^{r-1}}{dt^{r-1}} - (-1)^{r-1} t \right) \tAi_{r,k}(t)
	& =
	(-1)^{r-1} \frac{\theta^{k\frac{r-2}{2}}}{2\pi\iu} \int_{\tilde{C}_k} (w^{r-1} - t) e^{\frac{w^{r}}{r} - t w} dw \\
	& =
	(-1)^{r-1} \frac{\theta^{k\frac{r-2}{2}}}{2\pi\iu} \int_{\tilde{C}_k} \frac{d}{dw} e^{\frac{w^{r}}{r} - t w} dw\,,
\end{split}
\end{equation*}
and the right-hand side vanishes by definition of the contours $\tilde{C}_k$, since the integrand is exponentially decreasing along the contour. We can compute the asymptotic behaviour of the functions as $t \to \infty$ via the steepest descent method.

\begin{lemma}\label{lem:asmptA}
	The asymptotic behaviours as $t \to \infty$ of the hyper-Airy function and its derivatives are given by
	\begin{equation}
		\tAi_{r,k}^{(a)}(t)
		\sim
		\frac{(-\theta^k)^a}{\sqrt{2\pi(r-1)}}
		e^{
			- \frac{r-1}{r}
			\theta^k
			t^{\frac{r}{r-1}}
			}
		t^{-\frac{r-2a-2}{2(r-1)}}
		\Bigl(
			1 + O\bigl( t^{-\frac{r}{r-1}} \bigr)
		\Bigr)\,.
	\end{equation}
\end{lemma}

\begin{proof}
	With the change of variables $w = t^{1/(r-1)} z$, we find
	\begin{equation*}
		\tAi_{r,k}^{(a)}(t)
		=
		\frac{\theta^{k\frac{r-2}{2}}}{2\pi\iu} \int_{\tilde{C}_k} (-w)^a e^{\frac{w^{r}}{r} - t w} dw
		=
		\frac{\theta^{k\frac{r-2}{2}}}{2\pi\iu} t^{\frac{a+1}{r-1}}
		\int_{\tilde{\Gamma}_k} (-z)^a e^{t^{\frac{r}{r-1}} ( \frac{z^{r}}{r} - z )} dz\,.
	\end{equation*}
	Here $\tilde{\Gamma}_k$ is the Lefschetz thimble in the $z$-plane that passes through the critical point $z = \theta^k$. We can now apply the steepest descent method to the function
	\begin{equation*}
		I(s) = \int_{\Gamma} g(z) e^{s h(z)} dz\,,
	\end{equation*}
	with $s = t^{\frac{r}{r-1}}$, $g(z) = (-z)^a$, and $h(z) = \frac{z^{r}}{r} - z$. As a consequence, we find
	\begin{equation*}
		I(s)
		\sim
		g(z_k)
		\sqrt{\frac{2\pi}{s |h''(z_k)|}}
		e^{s h(z_k)} e^{\iu \vartheta(z_k)}
		\left(
			1
			+
			O\bigl( s^{-1} \bigr)
		\right),
	\end{equation*}
	where the phase $\vartheta$ is given by $\vartheta(z_k) = \frac{\pi - \arg(h''(z_k))}{2}$. A simple computation shows that $g(z_k) = (-\theta^k)^a$, $h(z_k) = - \frac{r-1}{r} \theta^k$, and $h''(z_k) = (r-1) \theta^{k(r-2)}$. Thus,
	\begin{equation*}
		I(s)
		\sim
		(-\theta^k)^a
		\sqrt{\frac{2\pi}{s (r-1)}}
		e^{ - \frac{r-1}{r} \theta^k s}
		e^{\iu (\frac{\pi}{2} - \frac{2\pi k(r-2)}{2(r-1)})}
		\left(
			1
			+
			O\bigl( s^{-1} \bigr)
		\right).
	\end{equation*}
	Taking into account the prefactor and the change of variables $s = t^{\frac{r}{r-1}}$, we find the claimed asymptotic behaviour.
\end{proof}

We can get the full asymptotic expansion by means of the differential equation, and the coefficients of the asymptotic expansion are expressed in terms of the polynomials $P_m(r,a)$ defined by \cref{eqn:PPZ:polys}.

\begin{proposition}\label{prop:hyper:airy}
	The asymptotic expansions as $t \to \infty$ of the hyper-Airy function and its derivatives are given by
	\begin{equation}
		\tAi_{r,k}^{(a)}(t)
		\sim
		\frac{(-\theta^k)^a}{\sqrt{2\pi(r-1)}}
		e^{
			- \frac{r-1}{r}
			\theta^k
			t^{\frac{r}{r-1}}
			}
		t^{-\frac{r-2a-2}{2(r-1)}}
		\sum_{m \ge 0}
			P_m(r,a) \left( \theta^{-k} \frac{t^{-\frac{r}{r-1}}}{r-1}  \right)^m.
	\end{equation}
\end{proposition}

\begin{proof}
	Let us set $u_{a} = \sqrt{2\pi(r-1)} \, \tAi_{r,k}^{(a)}$. From \cref{lem:asmptA}, we know that the asymptotic expansion of $u_{a}$ is of the form
	\begin{equation*}
		u_{a}(t)
		\sim
		(-\theta^k)^a
		e^{
			- \frac{r-1}{r}
			\theta^k
			t^{\frac{r}{r-1}}
			}
		t^{-\frac{r-2a-2}{2(r-1)}}
		\sum_{m \ge 0}
			p_m(r,a,k) t^{-m\frac{r}{r-1}}\,,
	\end{equation*}
	for some coefficients $p_m(r,a,k)$ with $p_0 = 1$. Thus, the asymptotic expansion of the derivative of $u_{a-1}$ for $a = 1,\dots,r-2$ reads
	\begin{equation*}
	\begin{split}
		\dot{u}_{a-1}(t)
		\sim
		(-\theta^k)^{a-1}
		e^{
			- \frac{r-1}{r}
			\theta^k
			t^{\frac{r}{r-1}}
			}
		t^{-\frac{r-2a-2}{2(r-1)}}
		\sum_{m \ge 0}
			\biggl(
				- \theta^k
				p_m(r,a-1,k)
				+ & \\
				+
				\frac{r}{r-1} \left( - m + \frac{1}{2} + \frac{a}{r} \right)
				p_{m-1}(r,a-1,k)
			\biggr)
			t^{-m\frac{r}{r-1}} & ,
	\end{split}
	\end{equation*}
	which equals
	\begin{equation*}
		u_{a}(t)
		\sim
		(-\theta^k)^{a} \,
		e^{
			- \frac{r-1}{r}
			\theta^k
			t^{\frac{r}{r-1}}
		} \,
		t^{-\frac{r-2a-2}{2(r-1)}}
		\sum_{m \ge 0}
			p_m(r,a,k) t^{-m\frac{r}{r-1}} \, .
	\end{equation*}
	In particular, we get the equality
	\begin{equation*}
		p_m(r,a-1,k) - \theta^{-k} \frac{r}{r-1} \left( - m + \frac{1}{2} + \frac{a}{r} \right) p_{m-1}(r,a-1,k)
		=
		p_m(r,a,k)\,,
	\end{equation*}
	for $a = 1,\dots,r-2$. On the other hand, the ODE is equivalent to $u_{r-1} = (-1)^{r-1} t \, u_0$, so that we easily find $p_m(r,r-1,k) = p_m(r,0,k)$. By setting $p_m(r,a,k) = (\theta^k(r-1))^{-m} P_m(r,a,k)$, we obtain the claimed recursion \eqref{eqn:PPZ:polys} for the coefficients $P_m(r,a,k)$. Notice that $P_m(r,a,k)$ is independent of $k$, so we can drop it from the notation.
\end{proof}

For later convenience, let us introduce the following notation for the asymptotic expansion of the hyper-Airy functions and their derivatives.

\begin{definition}
	Define the formal power series
	\begin{equation}
		\mathsf{A}_{r}^{(a)}(u)
		=
		\sum_{m \ge 0}
			P_m(r,a) \left( \frac{u}{r(r-1)}\right)^m,
	\end{equation}
	so that the asymptotic expansions as $t \to \infty$ of the hyper-Airy function and its derivatives can be written as
	\begin{equation}
		\tAi_{r,k}^{(a)}(t)
		\sim
		\frac{(-\theta^k)^a}{\sqrt{2\pi(r-1)}} \,
		e^{ - \frac{r-1}{r} \theta^k t^{\frac{r}{r-1}} } \,
		t^{-\frac{r-2a-2}{2(r-1)}} \,
		\mathsf{A}_{r}^{(a)}\bigl( \theta^{-k} \, r t^{- \frac{r}{r-1}} \bigr) \, .
	\end{equation}
\end{definition}

\section{The spectral curves}
\label{sec:TR}

As explained in \cref{subsec:TR}, to every compact spectral curve $\mc{S}$ one can associate a semisimple CohFT $\Omega_{g,n} = RTw_{g,n}$, constructed through the action of an $R$-matrix and a translation on a TFT, such that the expansion coefficients of the topological recursion correlators $\omega_{g,n}$ on a certain basis of differential forms coincide with the CohFT correlators:
\begin{equation}
	\int_{\overline{\mc{M}}_{g,n}} \Omega_{g,n}(e_{i_1} \otimes \cdots \otimes e_{i_n})
		\prod_{i=1}^n \psi_i^{k_i}\,.
\end{equation}
The main goal of this section is to apply the TR-CohFT correspondence to two global spectral curves, proving that the associated correlators are the Witten classes shifted along $v_{r-2}$ and $v_1$, and establish a connection with the generalised Airy functions.

\subsection{The spectral curve for the \texorpdfstring{$v_{r-2}$ shift}{shift along (0,...,0,r)}}

The goal of this section is to prove the following result.

\begin{theorem}\label{thm:SC:Wspin:hat}
	Let $\widehat{\mc{S}}_{r,\epsilon}$ be the $1$-parameter family of spectral curves on $\P^1$ given by
	\begin{equation}\label{eqn:SC:Wspin:hat}
			x(z) = V_r(z,\epsilon) \, ,
		 	\qquad\quad
		 	y(z) = z \, ,
		 	\qquad\quad
		 	\omega_{0,2}(z_1,z_2) = \frac{dz_1 dz_2}{(z_1 - z_2)^2} \, ,
	\end{equation}
	where $V_r$ is the $r$-th Lucas polynomial of the second kind. The CohFT associated to $\widehat{\mc{S}}_{r,\epsilon}$ is the shifted Witten $r$-spin class $\widehat{W}^{r,\epsilon}$. More precisely, the topological recursion correlators $\widehat{\omega}_{g,n}^{r,\epsilon}$ corresponding to the spectral curve $\widehat{\mc{S}}_{r,\epsilon}$ are
	\begin{equation}
	\begin{split}
		&
		\widehat{\omega}_{g,n}^{r,\epsilon}(z_1,\dots,z_n) = \\
		& \qquad
		=
		(-r)^{g-1}
		\sum_{a_1,\dots,a_n = 0}^{r-2}
		\int_{\overline{\mc{M}}_{g,n}}
				\widehat{W}^{r,\epsilon}_{g,n}(v_{a_1} \otimes\cdots\otimes v_{a_n})
			\prod_{i=1}^n \sum_{k_i \ge 0} \psi_i^{k_i} d\hat{\xi}^{k_i,a_i}(z_i) \, ,
	\end{split}
	\end{equation}
	and the differentials $d\hat{\xi}^{k,a}(z)$ are given by
	\begin{equation}
		\hat{\xi}^{a}(z)
		=
		- \frac{U_{r-1-a}(z,\epsilon)}{U_r(z,\epsilon)} \, ,
		\qquad\quad
		d\hat{\xi}^{k,a}(z)
		=
		d\left( \left(
				\frac{1}{r U_r(z,\epsilon)} \frac{d}{dz}
			\right)^k
			\hat{\xi}^{a}(z)
		\right) ,
	\end{equation}
	where $U_p$ is the $p$-th Lucas polynomial of the first kind.
\end{theorem}

To prove the theorem, let us compute the ingredients of the TR-CohFT correspondence. We refer to \cref{app:Lucas} for properties of the Lucas sequences that will be used in this section. We use the identification 
\[
	\lambda := \sqrt{\epsilon},
\] in order to avoid annoying square routes in the following. The ramification points of the spectral curve $\widehat{\mc{S}}_{r,
\epsilon}$ are given by
\begin{equation}
	x'(z) = r U_{r}(z,\lambda^2) = r \lambda^{r-1} \mc{U}_{r-1}\left(\frac{z}{2\lambda}\right) = 0\,,
\end{equation}
where $\mc{U}_s$ is the $s$-th Chebyshev polynomial of the second kind. The solutions are $\alpha_k = 2\lambda \cos(\frac{k\pi}{r})$ for $k = 1,\dots,r-1$, and we have $x(\alpha_k) = x_k = (-1)^k 2 \lambda^r$. We choose local coordinates such that
\begin{equation}
	x(z) - x_k = - \frac{\zeta_k^2(z)}{2}\,,
\end{equation}
and a determination of the square root such that
\begin{equation}
	\zeta_{k}(z)
	=
	\frac{1}{\sqrt{(-1)^{k}}} \frac{r \, \lambda^{\frac{r-2}{2}}}{\sqrt{2} \sin(\frac{k\pi}{r})}
	(z - \alpha_k) + O\bigl( (z - \alpha_k)^2 \bigr)\,.
\end{equation}
As a consequence, $\Delta^{k} = \sqrt{(-1)^{k}} \, \frac{\sqrt{2} \sin(\frac{k\pi}{r})}{r \lambda^{(r-2)/2}}$. Moreover, we choose the global constant $C = - \iu \sqrt{r} \lambda^{\frac{r-2}{2}}$, so that $h^k = \sqrt{\frac{2}{r}} \frac{ \sin(\frac{k\pi}{r})}{\sqrt{(-1)^{k+1}}}$. In this way, the TFT is given by
\begin{equation}
\begin{aligned}
	&
	V = \C\braket{e_1,\dots,e_{r-1}}\,,
	\qquad
	\eta(e_{k},e_{l}) = \delta_{k,l}\,,
	\qquad
	\bm{1} = \sqrt{\frac{2}{r}} \, \sum_{k=1}^{r-1} \frac{\sin(\tfrac{k\pi}{r})}{\sqrt{(-1)^{k+1}}} e_k\,, \\
	&
	\widehat{w}_{g,n}(e_{k_1}\otimes \cdots \otimes e_{k_n})
	=
	\delta_{k_1,\dots,k_n}
	\left( \sqrt{\frac{r}{2}} \frac{\sqrt{(-1)^{k+1}}}{\sin(\tfrac{k\pi}{r})} \right)^{2g-2+n}.
\end{aligned}
\end{equation}
Let us compute the other ingredients of the TR-CohFT correspondence.

\begin{lemma}
	In the canonical basis $(e_1,\dots,e_{r-1})$, we have the following:
	\begin{itemize}
		\item
		The auxiliary functions are given by
		\begin{equation}
			\xi^{k}(z)
			=
			- {\textstyle \sqrt{(-1)^{k}}} \, 
			\frac{\sqrt{2} \sin(\frac{k\pi}{r})}{r \lambda^{\frac{r-2}{2}}}
			\frac{1}{z - 2\lambda \cos(\frac{k\pi}{r})}\,.
		\end{equation}

		\item
		The $R$-matrix elements in the canonical basis (denoted $R^{-1}_{\textup{C}}$) are given by
		\begin{equation}
			R^{-1}_{\textup{C}}(u)_i^j
			=
			\frac{2}{r}
			\frac{\sqrt{(-1)^j}}{\sqrt{(-1)^{i}}}
			\sum_{p=0}^{r-2}
				\sin \left( \frac{(p+1)i \pi}{r} \right)
				\sin \left( \frac{(p+1)j \pi}{r} \right)
				\mathsf{B}_{r,p} \left( (-1)^{j+1} \frac{u}{\lambda^r} \right)  .
		\end{equation}

		\item
		The translation is given by $T(u) = u (\bm{1} - R^{-1}(u)\bm{1})$. Hence, the CohFT $\Omega_{g,n} = RT\widehat{w}_{g,n}$ obeys the flat unit axiom.
	\end{itemize}
\end{lemma}

\begin{proof} 
	We compute the auxiliary functions as
	\begin{equation*}
	\begin{split}
		\xi^{k}(z)
		& =
		\int^z \frac{B(z_0, \cdot)}{d\zeta_{k}(z_0)} \bigg|_{z_0 = \alpha_{k}}
		=
		\Delta^{k} \int^z \frac{dz}{(\alpha_{k} - z)^2}
		=
		-\frac{\Delta^{k}}{z - \alpha_k} \,.
	\end{split}
	\end{equation*}
	For the $R$-matrix (we drop the subscript ``C'' in the proof), we start by integrating by parts:
	\begin{equation*}
		R^{-1}(u)_{i}^{j}
		=
		- \sqrt{\frac{u}{2\pi}} \int_{\gamma_j} d\xi^{i} \, e^{\frac{1}{u}(x - x_j)}
		=
		\sqrt{\frac{u}{2\pi}} \int_{\gamma_j}
			\xi^{i} \, e^{\frac{1}{u}(x - x_j)}
			\frac{dx}{u}\,.
	\end{equation*}
	We perform now the change of variables $z = (\frac{u}{r})^{1/r}w$. A simple computation shows that
	\begin{equation*}
	\begin{split}
		\frac{x - x_j}{u}
		& =
		\frac{1}{u} V_r\left( (\tfrac{u}{r})^{1/r} w , \lambda^2 \right) - \frac{x_j}{u}
		=
		\frac{1}{r} V_r\left( w , \lambda^2 (\tfrac{r}{u})^{2/r} \right) - \frac{x_j}{u}\,
		,
		\\
		\frac{dx}{u}
		& =
		\frac{r}{u} U_{r}\left( (\tfrac{u}{r})^{1/r} w , \lambda^2 \right)
		\left( \frac{u}{r} \right)^{1/r} dw
		=
		U_r\left( w , \lambda^2 (\tfrac{r}{u})^{2/r} \right) dw \, .
	\end{split}
	\end{equation*}
	In both expressions, we used the quasi-homogeneity of type $(1,2)$ of the Lucas sequences (see \cref{table:Lucas}).	Moreover, $w$ runs along the Lefschetz thimble $\hat{C}_j$ passing through the critical point $w = 2 \lambda \left( \frac{r}{u} \right)^{1/r} \cos(\frac{j\pi}{r})$. As a consequence, we find
	\begin{equation*}
	\begin{split}
		R^{-1}(u)^{j}_{i}
		& =
		- \Delta^i e^{-\frac{x_j}{u}} \sqrt{\frac{u}{2\pi}}
		\int_{\hat{C}_j}
			\frac{ U_{r}(w,t) }{(\frac{u}{r})^{1/r} w - \alpha_i}
			e^{\frac{1}{r} V_r(w,t)}
			dw \\
		& =
		- \Delta^i e^{-\frac{x_j}{u}} \sqrt{\frac{u}{2\pi}} \left( \frac{r}{u} \right)^{\frac{1}{r}}
		\int_{\hat{C}_j}
			\frac{ U_{r}(w,t) }{w - (\frac{r}{u})^{1/r} \alpha_i}
			e^{\frac{1}{r} V_r(w,t)}
			dw \, ,
	\end{split}
	\end{equation*}
	where in the last equation we simply set $t^{1/2} = \lambda \left( \frac{r}{u} \right)^{1/r}$. As $(\frac{r}{u})^{1/r} \alpha_i = 2\sqrt{t} \cos(\frac{i\pi}{r})$, we can use \cref{lem:Uroot} to get
	\begin{equation*}
	\begin{split}
		& R^{-1}(u)^{j}_{i} = \\
		& =
		- \Delta^i e^{-\frac{x_j}{u}} \sqrt{\frac{u}{2\pi}} \left( \frac{r}{u} \right)^{\frac{1}{r}}
		\sum_{p=0}^{r-2} (-1)^{i+1} t^{\frac{r-p-2}{2}}
			\frac{\sin \bigl( \frac{(p+1)i \pi}{r} \bigr)}{\sin \left( \frac{i \pi}{r} \right)}
			\int_{\hat{C}_j}
				U_{p+1}(w,t)
				e^{\frac{1}{r} V_r(w,t)}
				dw \\
		& =
		2 \sqrt{\frac{\pi}{r}}
		e^{-\frac{(-1)^j 2 \lambda^r}{u}}
		\frac{\sqrt{(-1)^j}}{\sqrt{(-1)^{i}}}
		\sum_{p=0}^{r-2}
			t^{\frac{r-2p-2}{4}}
			\sin \left( \frac{(p+1)i \pi}{r} \right)
			\hAi_{r,p}^{j}(t) \, .
	\end{split}
	\end{equation*}
	Using the asymptotic expansion of the functions $\hAi_{r,p}^{j}$ derived in \cref{prop:LAasymp}, the above expression simplifies to
	\begin{equation*}
		R^{-1}(u)^{j}_{i}
		=
		\frac{2}{r}
		\frac{\sqrt{(-1)^j}}{\sqrt{(-1)^{i}}}
		\sum_{p=0}^{r-2}
			\sin \left( \frac{(p+1)i \pi}{r} \right)
			\sin \left( \frac{(p+1)j \pi}{r} \right)
			\ms{B}_{r,p} \left( (-1)^{j+1} r t^{-r/2} \right) \, .
	\end{equation*}
	We can now compute the translation: integrating by parts, we find
	\begin{equation*}
	\begin{split}
		T^k(u)
 		& =
 		u h^k + \iu \sqrt{\frac{u r}{2\pi}} \lambda^{\frac{r-2}{2}} \int_{\gamma_k} dy \, e^{\frac{1}{u}(x - x_k)} \\
 		& =
 		u h^k - \iu \sqrt{\frac{u r}{2\pi}} \lambda^{\frac{r-2}{2}} \int_{\gamma_k} y \, e^{\frac{1}{u}(x - x_k)} \frac{dx}{u} \, .
	\end{split}
	\end{equation*}
	Again, we perform the change of variables $z = (\frac{u}{r})^{1/r} w$, and the substitution $t^{1/2} = \lambda \left( \frac{r}{u} \right)^{1/r}$ to obtain
	\[
		T^k(u)
		=
		u h^k
		-
		\iu e^{-\frac{x_k}{u}} \frac{u}{\sqrt{2\pi}} t^{\frac{r-2}{4}}
		\int_{\hat{C}_k}
			w \, U_r(w,t)
			e^{\frac{1}{r} V_r(w,t)}
			dw \, .
	\]
	Using the identity of \cref{lem:Lid} and plugging in the asymptotic expansion of the functions $\hAi_{r,0}^k$ calculated in \cref{prop:LAasymp}, this simplifies to 
	\[
	\begin{split}
		T^k(u)
		& =
		u h^k
		+
		\iu e^{-\frac{x_k}{u}} \frac{u}{\sqrt{2\pi}} t^{\frac{r-2}{4}}
		\, 2\pi {\textstyle \sqrt{(-1)^k}} \;
		\hAi_{r,0}^k(t) \\
		& =
		u h^k
		+
		u
		\sqrt{\frac{2}{r}}
		{\textstyle \sqrt{(-1)^{k+1}}}
		\sin{ \left( \frac{k\pi}{r} \right) }
		\,
		\ms{B}_{r,0} \left( (-1)^{k+1} \frac{u}{\lambda^r} \right) .
	\end{split}
	\]
	The relation $T(u) = u (\bm{1} - R^{-1}(u)\bm{1})$ is equivalent to showing that
	\[
		\sum_{i=1}^{r-1} R^{-1}(u)^{j}_{i} \, h^i
		=
		- \sqrt{\frac{2}{r}}
		{\textstyle \sqrt{(-1)^{j+1}}}
		\sin{ \left( \frac{j\pi}{r} \right) }
		\,
		\ms{B}_{r,0} \left( (-1)^{j+1} \frac{u}{\lambda^r} \right),
	\]
	which follows from the trigonometric identity
	\begin{equation}\label{eqn:trig:id}
		\sum_{i=1}^{r-1} \sin \left( \frac{(p+1) i \pi}{r} \right) \sin \left( \frac{(q+1) i \pi}{r} \right)
		=
		\frac{r}{2} \, \delta_{p,q} \, . \qedhere
	\end{equation}
\end{proof}

The flat basis $(v_0,\dots,v_{r-2})$ in terms of the canonical basis  of the TFT, and the inverse base transformation to the canonical basis are given by the formulae
\begin{equation}
\begin{split}
	v_a &= \sqrt{\frac{2}{r}} \, \sum_{j=1}^{r-1} \frac{1}{\sqrt{(-1)^{j+1}}} \sin\left(\frac{(a+1)j\pi}{r}\right) e_{j}\,, \\
	e_j &= \sqrt{\frac{2}{r}} \, \sum_{a=0}^{r-2} {\textstyle \sqrt{(-1)^{j+1}}} \sin\left(\frac{(a+1)j\pi}{r}\right) v_a\,.
\end{split}
\end{equation}
It is easy to see that these transformations are inverse to each other using the trigonometric identity~\eqref{eqn:trig:id}.

\begin{proposition}\label{prop:CohFT:hat:TR}
	In the flat basis $(v_0,\dots,v_{r-2})$, the CohFT associated to the spectral curve $\widehat{\mc{S}}_{r,\epsilon}$ is 
	\begin{itemize}
		\item
		The pairing and the TFT are expressed as
		\begin{align}
			& \eta(v_a,v_b)
			=
			\delta_{a+b,r-2}\,, \\
			& \widehat{w}_{g,n}(v_{a_1} \otimes \cdots \otimes v_{a_n})
			=
			\left( \frac{r}{2} \right)^{g-1}\sum_{k=1}^{r-1} \frac{(-1)^{(k-1)(g-1)}\prod_{i=1}^n \sin \left( \frac{(a_i+1)k\pi}{r}\right)}{\sin^{2g-2+n}\left(\frac{k\pi}{r}\right)}.
		\end{align}
		 Moreover, the unit is given by $\bm{1} = v_0$.

		\item
		The $R$-matrix elements in the flat basis (denoted $R^{-1}_{\textup{F}}$) are given by
		\begin{equation}
			R_{\textup{F}}^{-1}(u)^{b}_{a}
			=
			\ms{B}^{\textup{even}}_{r,a} \left( \frac{u}{\lambda^r} \right) \delta_{b,a}
		  +
		  \ms{B}^{\textup{odd}}_{r,a} \left( \frac{u}{\lambda^r} \right) \delta_{b,r-2-a}
		  \, .
		\end{equation}
	\end{itemize}
\end{proposition}

\begin{proof}
	The formula for the pairing $ \eta(v_a,v_b)=\delta_{a+b,r-2} $ is easy to check using the trigonometric identity
	\begin{equation}\label{eqn:signed:trig:id}
		\sum_{j=1}^{r-1} (-1)^j \sin \left( \frac{(a+1) j \pi}{r} \right) \sin \left( \frac{(b+1) j \pi}{r} \right)
		=
		-\frac{r}{2} \, \delta_{a+b,r-2}\,.
	\end{equation}
	The TFT follows from an elementary change of basis computation. Let us now compute the $ R $-matrix in the flat basis,
	\begin{align*}
		R_{\textup{F}}^{-1}(u)^{b}_{a}
		&=
		\left( \frac{2}{r} \right)^{2}
		\sum_{i,j=1}^{r-1}
			\frac{\sqrt{(-1)^{j+1}}}{\sqrt{(-1)^{i+1}}}
			\sin \left( \frac{(a+1)i \pi}{r} \right)
			\sin \left( \frac{(b+1)j \pi}{r} \right) \\
		& \qquad\qquad \times
			\frac{\sqrt{(-1)^{j}}}{\sqrt{(-1)^{i}}}
			\sum_{p=0}^{r-2}
				\sin \left( \frac{(p+1)i \pi}{r} \right)
				\sin \left( \frac{(p+1)j \pi}{r} \right)
				\ms{B}_{r,p} \left( (-1)^{j+1} \frac{u}{\lambda^r} \right) \\
		&=
		- \left( \frac{2}{r} \right)^{2}
		\sum_{p=0}^{r-2} \sum_{j=1}^{r-1}
			\sin \left( \frac{(b+1)j \pi}{r} \right)
			\sin \left( \frac{(p+1)j \pi}{r} \right) \\
		& \qquad\qquad \times
			\Biggl(
			\sum_{j=1}^{r-1}
				\sin \left( \frac{(a+1)i \pi}{r} \right)
				\sin \left( \frac{(p+1)i \pi}{r} \right)
			\Biggr)		
				\ms{B}_{r,p} \left( (-1)^{j+1} \frac{u}{\lambda^r} \right) \\
		&=
		\frac{2}{r}
		\sum_{j=1}^{r-1}
			\sin \left( \frac{(b+1)j \pi}{r} \right)
			\sin \left( \frac{(a+1)j \pi}{r} \right)
			\ms{B}_{r,a} \left( (-1)^{j+1} \frac{u}{\lambda^r} \right) \\
	  &=
	  \ms{B}^{\textup{even}}_{r,a} \left( \frac{u}{\lambda^r} \right) \delta_{b,a}
	  +
	  \ms{B}^{\textup{odd}}_{r,a} \left( \frac{u}{\lambda^r} \right) \delta_{b,r-2-a}
	  \, .
	\end{align*}
	In the second line, we used the equation
	\[
		\frac{\sqrt{(-1)^{j+1}}}{\sqrt{(-1)^{i+1}}}
		\frac{\sqrt{(-1)^{j}}}{\sqrt{(-1)^{i}}}
		=
		\frac{\iu}{\iu}
		=
		1 \, ,
	\]
	and simplified further using the aforementioned trigonometric identities.
\end{proof}

We are now ready to prove the main result of the section.

\begin{proof}[{Proof of \cref{thm:SC:Wspin:hat}}]
	From the TR-CohFT correspondence and \cref{prop:CohFT:hat:TR} above, we find that the topological recursion correlators $\omega_{g,n}(z_1,\dots,z_n)$ associated to the spectral curve $\widehat{\mc{S}}_{r,\epsilon}$ compute the intersections of $\widehat{W}^{r,\epsilon}$ and $\psi$-classes in the canonical basis, up to some overall factors that are missing in the TFT in  \cref{prop:CohFT:hat:TR}. To be precise, we have
	\[
		\frac{( - \iu\sqrt{r} \lambda^{\frac{r-2}{2}} )^{2g-2+n}}{ \lambda^{(r-2)(g-1) + |a|} }
		\sum_{i_1,\dots,i_n = 1}^{r-1} \int_{\overline{\mc{M}}_{g,n}}
			\widehat{W}^{r,\epsilon}_{g,n}(e_{i_1} \otimes\cdots\otimes e_{i_n})
			\prod_{j=1}^n \sum_{k_j \ge 0}
			\psi_j^{k_j} d\xi^{k_j,i_j}(z_j) \, .
	\]
	Here the differentials $d\xi^{k,i}$ are given by
	\[
	\begin{split}
		\xi^i(z)
		& =
		- {\textstyle \sqrt{(-1)^{k}}} \,
		\frac{\sin(\frac{i\pi}{r})}{\sqrt{(-1)^{i+1}} r \lambda^{\frac{r-2}{2}}}
		\frac{1}{z - 2\lambda \cos(\frac{i\pi}{r})} \, ,\\
		d\xi^{k,i}(z)
		& =
		d\biggl(
			\left( \frac{1}{r U_r(z,\lambda^2)} \frac{d}{dz} \right)^k \xi^i(z)
		\biggr)\,.
	\end{split}
	\]
	Changing to the flat basis amounts to replacing the function $\xi^i(z)$ by the following linear combination
	\begin{multline*}
		\sqrt{\frac{2}{r}} \sum_{i = 1}^{r-1} {\textstyle \sqrt{(-1)^{i+1}}} \sin\left( \frac{(a+1)i\pi}{r} \right) \xi^i(z) = \\
		=
		\frac{- \iu}{\sqrt{r} \lambda^{\frac{r-2}{2}}}
		\frac{2}{r} \sum_{i = 1}^{r-1} \sin\left( \frac{(a+1)i\pi}{r} \right) \sin\left( \frac{i\pi}{r} \right) \frac{1}{z - 2\lambda \cos(\frac{i\pi}{r})}\,.
	\end{multline*}
	Denote the above function by $\frac{\iu \lambda^a}{\sqrt{r} \lambda^{\frac{r-2}{2}}} \hat{\xi}^a(z)$. Thus, we find
	\[
	\begin{split}
		&
		\widehat{\omega}_{g,n}^{r,\epsilon}(z_1,\dots,z_n) = \\
		& \quad =
		\frac{( - \iu\sqrt{r} )^{2g-2}}{\lambda^{|a|}}
		\sum_{a_1,\dots,a_n = 0}^{r-2} \int_{\overline{\mc{M}}_{g,n}}
			\widehat{W}^{r,\epsilon}_{g,n}(v_{a_1} \otimes\cdots\otimes v_{a_n})
			\prod_{j=1}^n \sum_{k_j \ge 0}
			\psi_j^{k_j} \lambda^{a_j} d\hat{\xi}^{k_j,a_j}(z_j) \\
		& \quad =
		( -r )^{g-1}
		\sum_{a_1,\dots,a_n = 0}^{r-2} \int_{\overline{\mc{M}}_{g,n}}
			\widehat{W}^{r,\epsilon}_{g,n}(v_{a_1} \otimes\cdots\otimes v_{a_n})
			\prod_{j=1}^n \sum_{k_j \ge 0}
			\psi_j^{k_j} d\hat{\xi}^{k_j,a_j}(z_j) \, .
	\end{split}
	\]
	 To conclude, let us simplify the expression for the function $\hat{\xi}^a(z)$. We have
	\[
	\begin{split}
		\hat{\xi}^a(z)
		& =
		\frac{-2}{r \lambda^a} \sum_{i = 1}^{r-1}
			\sin\left( \frac{(a+1)i\pi}{r} \right) \sin\left( \frac{i\pi}{r} \right)
			\frac{1}{z - 2\lambda \cos(\frac{i\pi}{r})} \\
		& =
		\frac{-2}{r \lambda^a U_{r}(z,\lambda^2)}
		\sum_{i = 1}^{r-1}
			\sin\left( \frac{(a+1)i\pi}{r} \right) \sin\left( \frac{i\pi}{r} \right)
			\frac{U_{r}(z,\lambda^2)}{z - 2\lambda \cos(\frac{i\pi}{r})}\,.
	\end{split}
	\]
	We can now apply \cref{lem:Uroot} and simplify using the trigonometric identity \eqref{eqn:signed:trig:id} to get
	\[
	\begin{split}
		& \hat{\xi}^a(z) = \\
		& =
		\frac{- 2}{r \lambda^a U_{r}(z,\lambda^2)}
		\sum_{i = 1}^{r-1} \sum_{p=0}^{r-2}
			(-1)^{i+1} \lambda^{r-p-2}
			\sin\left( \frac{(a+1)i\pi}{r} \right) \sin\left( \frac{(p+1)i\pi}{r} \right)
			U_{p+1}(z,\lambda^2) \\
		& =
		- \frac{U_{r-1-a}(z,\lambda^2)}{U_{r}(z,\lambda^2)}\, .
	\end{split}
	\]
	This concludes the proof.
\end{proof}

\subsection{The spectral curve for the \texorpdfstring{$v_{1}$ shift}{shift along (0,r,...,0)}}

The goal of this section is to prove the following result.

\begin{theorem}\label{thm:SC:Wspin:tilde}
	Let $\widetilde{\mc{S}}_{r,\epsilon}$ be the $1$-parameter family of spectral curves on $\P^1$ given by
	\begin{equation}\label{eqn:SC:Wspin:tilde}
		x(z) = z^r - r \epsilon z\,,
		\qquad\quad
		y(z) = z\,,
		\qquad\quad
		\omega_{0,2}(z_1,z_2) = \frac{dz_1 dz_2}{(z_1 - z_2)^2} \, .
	\end{equation}
	The CohFT associated to $\widetilde{\mc{S}}_{r,\epsilon}$ is the shifted Witten $r$-spin class $\widetilde{W}^{r,\epsilon}$. More precisely, the topological recursion correlators $\widetilde{\omega}_{g,n}^{r,\epsilon}$ corresponding to the spectral curve $\widetilde{\mc{S}}_{r,\epsilon}$ are
	\begin{equation}\label{eqn:correl:shifted:tilde}
	\begin{split}
		&
		\widetilde{\omega}_{g,n}^{r,\epsilon}(z_1,\dots,z_n) = \\
		& \qquad
		=
		(-r)^{g-1}
		\sum_{a_1,\dots,a_n = 0}^{r-2}
		\int_{\overline{\mc{M}}_{g,n}}
			\widetilde{W}^{r,\epsilon}_{g,n}(v_{a_1} \otimes\cdots\otimes v_{a_n})
			\prod_{i=1}^n \sum_{k_i \ge 0} \psi_i^{k_i} d\tilde{\xi}^{k_i,a_i}(z_i) \, ,
	\end{split}
	\end{equation}
	and the differentials $d\tilde{\xi}^{k,a}(z)$ are given by
	\begin{equation}
		\tilde{\xi}^{a}(z)
		=
		- \frac{z^{r-2-a}}{z^{r-1} - \epsilon} \, ,
		\qquad\quad
		d\tilde{\xi}^{k,a}(z)
		=
		d\left( \left(
				\frac{1}{r (z^{r-1} - \epsilon)} \frac{d}{dz}
			\right)^k
			\tilde{\xi}^{a}(z)
		\right) .
	\end{equation}
\end{theorem}

Again, in order to prove the theorem we are going to compute the ingredients of the TR-CohFT correspondence. Throughout this section, we will use the parameter $ \lambda $ instead of $ \epsilon $, defined as
\[
	\lambda  := \epsilon^{1/(r-1)}
\] in order to avoid annoying $ (r-1) $th roots appearing. The ramification points of the spectral curve $\widetilde{\mc{S}}_{r,\epsilon}$ are given by
\begin{equation}
	x'(z) = r z^{r-1}- r \lambda^{r-1} = 0\,.
\end{equation}
It is solved by $\alpha_k = \theta^k \lambda$ for $k = 0,\dots,r-2$, where $\theta = e^{\frac{2\pi\iu}{r-1}}$. We have $x(\alpha_k) = x_k = -(r-1) \theta^k \lambda^r$. We choose local coordinates such that
\begin{equation}
	x(z) - x_k = - \frac{\zeta_k^2(z)}{2}\,,
\end{equation}
and a determination of the square root such that
\begin{equation}
	\zeta_{k}(z)
	=
	- \iu \sqrt{r(r-1)} (\theta^k \lambda)^{\frac{r-2}{2}}
	(z - \alpha_k) + O\bigl( (z - \alpha_k)^2 \bigr)\,.
\end{equation}
As a consequence, $\Delta^{k} = \frac{\iu}{\sqrt{r(r-1)} (\theta^k \lambda)^{\frac{r-2}{2}}}$. Moreover, we choose the global constant $C = -\iu \sqrt{r} \lambda^{\frac{r-2}{2}}$, so that $h^k = \frac{\theta^{-k \frac{r-2}{2}}}{\sqrt{r-1}}$. In this way, the TFT (expressed in the canonical basis) is given by
\begin{equation} 
\begin{aligned}
	& V = \C\braket{e_0,\dots,e_{r-2}}\,,
	\qquad
	\eta(e_{k},e_{l}) = \delta_{k,l}\,,
	\qquad
	\bm{1} = \frac{1}{\sqrt{r-1}} \sum_{k=0}^{r-2} \theta^{- k\frac{r-2}{2}} e_k\,, \\
	& \widetilde{w}_{g,n}(e_{k_1}\otimes \cdots \otimes e_{k_n})
	=
	\delta_{k_1,\dots,k_n} (r-1)^{g-1+\frac{n}{2}} \theta^{k(r-2)(g-1+\frac{n}{2})}\,.
\end{aligned}
\end{equation}
Let us compute the other ingredients of the TR-CohFT correspondence.

\begin{lemma}
	In the canonical basis $(e_0,\dots,e_{r-2})$, we have the following:
	\begin{itemize}
		\item
		The auxiliary functions are given by
		\begin{equation}
			\xi^{k}(z)
			=
			\frac{-\iu}{\sqrt{r(r-1)} (\theta^k \lambda)^{\frac{r-2}{2}}} \frac{1}{z - \theta^k \lambda}\,.
		\end{equation}

		\item
		The $R$-matrix elements in the canonical basis (denoted $R^{-1}_{\textup{C}}$) are given by
		\begin{equation}
			R^{-1}_{\textup{C}}(u)_i^j
			=
			\frac{1}{r-1}
			\sum_{s=0}^{r-2}
				\theta^{(i-j)\frac{r-2s-2}{2}}
				\mathsf{A}_r^{(s)} \left( \theta^{-j} \frac{u}{\lambda^r} \right)  .
		\end{equation}

		\item
		The translation is given by $T(u) = u (\bm{1} - R^{-1}(u)\bm{1})$. Hence, the CohFT $\Omega_{g,n} = RT\widetilde{w}_{g,n}$ obeys the flat unit axiom.
	\end{itemize}
\end{lemma}

\begin{proof}
	The auxiliary functions are computed as before.	For the $R$-matrix (again, we drop the subscript ``C'' in the proof), we start by integrating by parts:
	\begin{equation*}
		R^{-1}(u)_{i}^{j}
		=
		- \sqrt{\frac{u}{2\pi}} \int_{\gamma_j} d\xi^{i} \, e^{\frac{1}{u}(x - x_j)}
		=
		\sqrt{\frac{u}{2\pi}} \int_{\gamma_j}
			\xi^{i} \, e^{\frac{1}{u}(x - x_j)}
			\frac{dx}{u}\,.
	\end{equation*}
	We perform now the change of variables $z = (\frac{u}{r})^{1/r}w$. We have the relations
	\begin{equation*}
		\frac{x - x_j}{u}
		=
		\frac{w^r}{r} - \lambda^{r-1} \left( \frac{u}{r} \right)^{\frac{1}{r}-1} w - \frac{x_j}{u}\,
		,
		\qquad\quad
		\frac{dx}{u}
		=
		\left(
			w^{r-1} - \lambda^{r-1} \left( \frac{u}{r} \right)^{\frac{1}{r}-1}
		\right)
		dw\,.
	\end{equation*}
	Moreover, $w$ runs along the Lefschetz thimble $\tilde{C}_j$ passing through the critical point $w = \theta^k \lambda (\frac{r}{u})^{1/r}$. As a consequence, we find
	\begin{equation*}
	\begin{split}
		& R^{-1}(u)^{j}_{i}
		=
		-\Delta^i e^{- \frac{x_j}{u}} \sqrt{\frac{u}{2\pi}}
		\int_{\tilde{C}_j}
			\frac{w^{r-1} - \lambda^{r-1} \left( \frac{u}{r} \right)^{\frac{1}{r}-1}}{(\frac{u}{r})^{1/r} w - \alpha_i}
			e^{
				\frac{w^r}{r} - \lambda^{r-1} \left( \frac{u}{r} \right)^{\frac{1}{r}-1} w
			}
			dw \\
		& \quad =
		-\Delta^i e^{- \frac{x_j}{u}} \sqrt{\frac{u}{2\pi}} \left( \frac{u}{r} \right)^{-\frac{1}{r}}
		\sum_{s=0}^{r-2}
		\left( \alpha_i \left( \frac{u}{r} \right)^{-\frac{1}{r}} \right)^{r-2-s}
		\int_{\tilde{C}_j}
			w^s
			e^{
				\frac{w^r}{r} - \lambda^{r-1} \left( \frac{u}{r} \right)^{\frac{1}{r}-1} w
			}
			dw\,.
	\end{split}
	\end{equation*}
	We recognise here the hyper-Airy functions and their derivatives, after the identification $t^{\frac{1}{r-1}} = \lambda^{r-1} \left( \frac{r}{u} \right)$. Expressing the $R$-matrix elements in terms of the $t$-variable, we find
	\begin{equation*}
	\begin{split}
		R^{-1}(u)^{j}_{i}
		& =
		e^{\frac{r-1}{r} \theta^j t^{\frac{r}{r-1}}}
		\sqrt{\frac{2\pi}{r-1}}
		t^{-\frac{r-2}{2(r-1)}}
		\theta^{-i\frac{r-2}{2}}
		\sum_{s=0}^{r-2}
			\frac{ ( \theta^i t^{\frac{1}{r-1}} )^{r-2-s} }{ \theta^{j\frac{r-2}{2}} }
			\left( -\frac{d}{dt} \right)^s \tAi_{r,j}(t) \\
		& =
		\frac{1}{r-1}
		\sum_{s=0}^{r-2}
			\theta^{(i-j)\frac{r-2s-2}{2}}
			\mathsf{A}_r^{(s)}\left( \theta^{-j} \frac{u}{\lambda^r} \right).
	\end{split}
	\end{equation*}
	We can now compute the translation: integrating by parts, we find
	\begin{equation*}
	\begin{split}
		T^k(u)
		& =
		u h^k + \iu \sqrt{\frac{u r}{2\pi}} \lambda^{\frac{r-2}{2}} \int_{\gamma_k} dy \, e^{\frac{1}{u}(x - x_k)} \\
 		& =
 		u h^k - \iu \sqrt{\frac{u r}{2\pi}} \lambda^{\frac{r-2}{2}} \int_{\gamma_k} y \, e^{\frac{1}{u}(x - x_k)} \frac{dx}{u} \, .
	\end{split}
	\end{equation*}
	As before, we perform the change of variables $z = (\frac{u}{r})^{1/r}w$ and set $t^{\frac{1}{r-1}} = \lambda^{r-1} \left( \frac{r}{u} \right)$ to get
	\begin{equation*}
	\begin{split}
		T^k(u)
		& =
		u h^k
		-
		\iu u e^{-\frac{x_k}{u}} \sqrt{\frac{1}{2\pi}} t^{\frac{r-2}{2(r-1)}}
		\int_{\tilde{C}_k}
			\left(
				w^{r} - t w
			\right)
			e^{
					\frac{w^r}{r} - t w
				}
			dw \\
 		& =
 		u h^k
		-
		u e^{-\frac{x_k}{u}} \sqrt{2\pi} t^{\frac{r-2}{2(r-1)}}
		\theta^{-k\frac{r-2}{2}}
		\Bigl(
			(-1)^{r-1} \tAi_{r,k}^{(r)}(t)
			-
			t \tAi_{r,k}'(t)
		\Bigr) .
	\end{split}
	\end{equation*}
	From the relation $\tAi_{r,k}^{(r)}(t) = (-1)^{r-1} \bigl( \tAi_{r,k}(t) + t \, \tAi_{r,k}'(t) \bigr)$, we find
	\begin{equation*}
	\begin{split}
		T^k(u)
		& =
		u h^k
		-
		u e^{-\frac{x_k}{u}} \sqrt{2\pi} t^{\frac{r-2}{2(r-1)}}
		\theta^{-k\frac{r-2}{2}}
		\tAi_{r,k}(t) \\
		& =
		u h^k
		-
		\frac{u}{\sqrt{r-1}} 	\theta^{-k\frac{r-2}{2}}  \mathsf{A}_{r}^{(0)}\left( \theta^{-k} \frac{u}{\lambda^r} \right).
	\end{split}
	\end{equation*}
	One can easily check that the relation $T(u) = u (\bm{1} - R^{-1}(u)\bm{1})$ is satisfied.
\end{proof}

\begin{remark}
	  Givental expressed the $R$-matrix as certain integral representations involving the miniversal deformation of the $A_{r-1}$ singularity in~\cite[sections 5 and 6]{Giv03}. It may be possible to derive the relation between the differential \cref{eqn:hyperAiry}  that the  hyperAiry functions satisfy and the $ R $-matrix for the $ v_1 $-shift using this result.
\end{remark}

The flat basis $(v_0,\dots,v_{r-2})$ in terms of the canonical basis  of the TFT, and the inverse base transformation to the canonical basis are given by the formulae
\begin{equation}
\begin{split}
	v_a &= \frac{1}{\sqrt{r-1}} \sum_{k=0}^{r-2} \theta^{-k \frac{(r-2)}{2} (2a+1)} e_k\, ,  \\
	e_k &=\frac{1}{\sqrt{r-1}}\sum_{a=0}^{r-2}\theta^{k \frac{(r-2)}{2}(2a+1)} v_a\, .
\end{split}
\end{equation}

\begin{proposition}\label{prop:CohFT:tilde:TR}
	In the flat basis $(v_0,\dots,v_{r-2})$, the CohFT associated to the spectral curve $\widetilde{\mc{S}}_{r,\epsilon}$ is expressed as
	\begin{itemize}
		\item
		The pairing and the TFT are expressed as
		\begin{equation}
			\eta(v_{a},v_{b})
			=
			\delta_{a+b,r-2} \, ,
			\qquad
			\widetilde{w}_{g,n}(v_{a_1} \otimes \cdots \otimes v_{a_n})
			=
			(r-1)^g \cdot \delta \, ,
		\end{equation}
		where $\delta$ is equal to $1$ if $r-1$ divides $g-1-|a|$ and $0$ otherwise. Moreover, the unit is given by $\bm{1} = v_0$.

		\item
		The $R$-matrix elements in the flat basis (denoted $R^{-1}_{\textup{F}}$) are given by
		\begin{equation}
			R_{\textup{F}}^{-1}(u)^{b}_{a}
			=
			\sum_{\substack{m \geq 0\\ a-b \equiv m \pmod{r-1}}} P_m(r,a) \left( \frac{1}{r(r-1)} \frac{u}{\lambda^r} \right)^{m} .
		\end{equation}
	\end{itemize}
\end{proposition}
\begin{proof}
	We present a proof for the $R$-matrix elements:
	\[
	\begin{split}
		R_{\textup{F}}^{-1}(u)^{b}_{a}
		& =
		\sum_{i,j,s = 0}^{r-2}
			\frac{1}{(r-1)^2}
			\theta^{-i\frac{r-2}{2}(2a+1)}
			\theta^{j\frac{r-2}{2}(2b+1)}
			\theta^{(i-j)\frac{r-2s-2}{2}}
			\mathsf{A}_r^{(s)}\left( \theta^{-j} \frac{u}{\lambda^r} \right) \\
		& =
		\sum_{m \geq 0}
			P_m(r,a) \left( \frac{1}{r(r-1)} \frac{u}{\lambda^r} \right)^{m} \frac{1}{r-1}
			\sum_{i=0}^{r-2} \theta^{i(a-b-m)} \\
		& =
		\sum_{\substack{m \geq 0\\ a-b \equiv m \pmod{r-1}}} P_m(r,a) \left( \frac{1}{r(r-1)} \frac{u}{\lambda^r} \right)^{m} .
	\end{split}
	\]
	The other computations are similar.
\end{proof}

We have computed all the necessary ingredients to prove the main result of this section.

\begin{proof}[{Proof of \cref{thm:SC:Wspin:tilde}}]
	The proof is completely parallel to that of \cref{thm:SC:Wspin:hat}.
\end{proof}

\begin{remark}\label{rem:DS}
	In this remark, we compare the results of this section to the results of \cite{DNOS19} concerning the shifted Witten class. The authors identify the spectral curve corresponding to the shifted Witten classes for almost all semi-simple shifts, as the following curve 
	\begin{equation}\label{eq:curve}
		x = z^r + \alpha_1 z^{r-2} + \cdots + \alpha_{r-1} \,,
		\quad
		y(z) = z\,,
		\quad
		\omega_{0,2}(z_1,z_2) = \frac{dz_1 dz_2}{(z_1 - z_2)^2} \, .
	\end{equation}
	For a point $ (\alpha_1, \cdots, \alpha_{r-1}) $ such that $ x(z) $ has $ r-1 $ distinct branch points, \cite[theorem~7.1]{DNOS19} states that the corresponding CohFT is the shift of the Witten class along the direction $ \tau $, which is defined as follows. Let $k$ be the series $k = x(z)^{1/r} = z + \frac{\alpha_1}{r}z^{-1} + O(z^{-2})$. Define $ \tau_a $ for $ 0 \leq a \leq r-2 $ as the first $r-1$ non-trivial coefficients of the inverse function expansion
	\begin{equation}\label{eq:flatdef}
		z = k + \frac{1}{r} \left( \frac{\tau_{r-2}}{k} + \cdots + \frac{\tau_{0}}{k^{r-1}} \right) + O(k^{-r-1}),
	\end{equation}
	which defines the flat coordinates for the $ A_{r-1} $ Dubrovin--Frobenius manifold\footnote{The definition of the flat coordinates in equation (7-11) of \cite{DNOS19} is incorrect and the right definition is the one here. This does not affect the rest of the proof of theorem 7.3 in \textit{loc.~cit.} } \cite[section~5.1]{Dub04}. Then, the point $ \tau $ is defined as
	\begin{equation}
		\tau = \sum_{a=0}^{r-2} \tau_a v_a.
	\end{equation}
	We note that it is easy to check that the spectral curves we have found do correspond to the corresponding claimed shifts using equation~\eqref{eq:flatdef}. However, theorem~7.1 of \cite{DNOS19} does not apply to the spectral curve $ \widehat{\mathcal{S}}_{r,\epsilon} $ corresponding to the $ v_{r-2} $-shift as the branch points of $ x $ are not pairwise distinct, and thus our \cref{thm:SC:Wspin:hat} is not immediately covered by the results of \textit{loc.~cit.} 
\end{remark}

\section{Witten's \texorpdfstring{$r$}{r}-spin conjecture}
\label{sec:Witten:conj}

In this section, we are interested in understanding the limits of the correlators $\widetilde{\omega}_{g,n}^{r,\epsilon}$ constructed by the topological recursion in the previous section, and the shifted Witten class, as $ \epsilon \to 0 $. We will prove that the $ r $-Airy curve computes the descendants of the Witten $ r $-spin class and use that to  prove Witten's $ r $-spin conjecture.

\subsection{The \texorpdfstring{$ \epsilon \to 0 $}{} limit}

For the shifted Witten class, we can use the following theorem of Pandharipande--Pixton--Zvonkine \cite{PPZ19}. 

\begin{proposition}[{\cite[theorems~8, 9]{PPZ19}}]\label{prop:limit:cohft}
	The constant coefficient in $\epsilon$ of both shifted Witten classes $\widehat{W}^{r,\epsilon}_{g,n}$ and $\widetilde{W}^{r,\epsilon}_{g,n}$ is the Witten class $W^{r}_{g,n}$:
	\begin{equation}\label{eqn:Witten:limit}
		\lim_{\epsilon \to 0} \widehat{W}^{r,\epsilon}_{g,n}(v_{a_1} \otimes\cdots\otimes v_{a_n})
		=
		\lim_{\epsilon \to 0} \widetilde{W}^{r,\epsilon}_{g,n}(v_{a_1} \otimes\cdots\otimes v_{a_n})
		=
		W^r_{g,n}(v_{a_1} \otimes\cdots\otimes v_{a_n}) \, .
	\end{equation}
	Moreover, the parts of $\widehat{W}^{r,\epsilon}_{g,n}(v_{a_1} \otimes\cdots\otimes v_{a_n})$ and $\widetilde{W}^{r,\epsilon}_{g,n}(v_{a_1} \otimes\cdots\otimes v_{a_n})$ of degree higher than $D^{r}_{g,a}$ vanish.
\end{proposition}

Now, we would like to understand the limit on the topological recursion side. We claim that taking the limit commutes with the topological recursion in the following sense. The limit as $ \epsilon \to 0 $ of both spectral curves $ \widetilde{\mc{S}}_{r,\epsilon}$ and $ \widehat{\mc{S}}_{r,\epsilon}$ gives the following spectral curve $\mc{S}_r$ on $\P^1 $, which is known in the literature as the \textit{$r$-Airy spectral curve}:
\begin{equation}
	x(z) = z^r \, ,
	\qquad\quad
	y(z) = z \, ,
	\qquad\quad
	\omega_{0,2}(z_1,z_2) = \frac{dz_1 dz_2}{(z_1-z_2)^2} \, .
\end{equation}
Then, we claim that the correlators constructed by topological recursion on $\mc{S}_r$ coincide with the limit as $ \epsilon \to 0 $ of the correlators constructed by topological recursion on $\widehat{\mc{S}}_{r,\epsilon}$ and $\widetilde{\mc{S}}_{r,\epsilon}$. Below we will prove the claim for a sufficiently general class of spectral curves and show that one of the spectral curves we are interested in fits into this class: $\widetilde{\mc{S}}_{r,\epsilon}$. This is enough for our purposes, since both $\widehat{\mc{S}}_{r,\epsilon}$ and $\widetilde{\mc{S}}_{r,\epsilon}$ allow us to access the same limit (see \cref{rmk-limit} for comments on the proof of the claim for the other spectral curve $\widehat{\mc{S}}_{r,\epsilon}$). In this section, we make the dependence on $\epsilon$ explicit in the topological recursion correlators, in order to avoid any confusion. 

A $1$-parameter family of spectral curves on $\P^1$ indexed by $\epsilon\in\C$ is the data, for each $\epsilon\in\C$, of a spectral curve $\mc{S}_{\epsilon}=(\Sigma, x_{\epsilon},y_{\epsilon},\omega_{0,2}^{\epsilon})$, where the Riemann surface $\Sigma$ is common for all $\epsilon$. The functions $x_{\epsilon}$ and $y_{\epsilon}$ are allowed to vary, as well as the bidifferential $\omega_{0,2}^{\epsilon}$. 

\begin{proposition}\label{prop:limit:tr}
	Let $\mc{S}_{\epsilon}$ be a family of spectral curves indexed by $\epsilon\in\C$ such that in a neighbourhood of $\epsilon=0$, they satisfy the following assumptions.
	\begin{enumerate}
		\item\label{itm:first} $\mc{S}_{\epsilon}$ is defined by an algebraic equation linear in $x_{\epsilon}$:
		\begin{equation}
			P_{\epsilon}(x_{\epsilon},y_\epsilon)=A_{\epsilon}(y) + x_{\epsilon}\,B_{\epsilon}(y_\epsilon) = 0 \, ,
		\end{equation}
		where $A_{\epsilon}(y_\epsilon)$, $B_{\epsilon}(y_\epsilon)$ are polynomials in $y_\epsilon$ and $\epsilon$. The Newton's polygon of this algebraic equation has no interior point, so the Riemann surface $\Sigma$ is of genus $0$. As a consequence, the bidifferentials $\omega_{0,2}^{\epsilon}$ do not depend on $\epsilon$: $\omega_{0,2}(z_1,z_2) = \frac{dz_1 dz_2}{(z_1-z_2)^2}$. 

		\item For $\epsilon \neq 0$, $\mc{S}_{\epsilon}$ has $r-1$ simple ramification points $\alpha_1, \dots, \alpha_{r-1}$, while $\mc{S}_0$ has a single ramification point of degree $r-1$ and is admissible in the sense of \cite{BBCCN23}. In addition, we assume that the branch points are distinct, i.e., that $ x_\epsilon(\alpha_i) \neq  x_\epsilon(\alpha_j) $ for any $ i \neq j $.

		\item The multidifferentials $\omega_{g,n}^{\epsilon}(z_1,\dots,z_n)$ produced by topological recursion admit limits as $\epsilon\to 0$:
		\begin{equation}
			\omega_{g,n}(z_1,\dots,z_n)\coloneqq \lim_{\epsilon \to 0} \, \omega_{g,n}^{\epsilon}(z_1,\dots,z_n) \, .
		\end{equation}
	\end{enumerate}
	Then the multidifferentials $\omega_{g,n}(z_1,\dots,z_n)$ satisfy the local topological recursion \eqref{eqn:local:TR} applied to the spectral curve $\mc{S}_0$. 
\end{proposition}

\begin{remark}\label{rem:limits}
	Before proving the statement, we remark that the above proposition does not hold  if we completely drop assumption \eqref{itm:first}. It is easy to find examples of families of spectral curves, where the limit of the correlators \textit{does not}~(!) coincide with the correlators of the limit spectral curve. However, all three assumptions can be relaxed and a more general statement than the one we prove here will appear in \cite{BBCKS23}. Thus, we content ourselves with the statement above that suffices for our purposes.
\end{remark}

\begin{proof}
	Let $U \subset \C$ be a neighbourhood of $\epsilon=0$ for which the assumptions of the proposition hold. For later use, denote by $p_0(z)$ the coefficient of $y^r$ in $P_{\epsilon}(x_{\epsilon},y_\epsilon)$. 

	First, for $\epsilon\in U$, topological recursion on $\mc{S}_{\epsilon}$ produces symmetric multidifferentials $\omega_{g,n}^{\epsilon}(z_1,\dots,z_n)$ (for $ \epsilon = 0 $ this is a consequence of the admissibility condition and \cite[theorem~E]{BBCCN23}). As a result, the following recursive formula, known as the \emph{global topological recursion} (see \cref{rem:ltog}), makes sense for all $\mc{S}_{\epsilon}$ with $\epsilon\in U$:
	\begin{equation}\label{eqn:global:TR}
	\begin{split}
		\omega_{g,n}^{\epsilon}(z_1,\dots,z_n)
		& =
		\sum_{i=1}^{r-1} \Res_{z=\alpha_i} \sum_{\substack{I\subset [r-1] \\ I \neq \varnothing}}
			\frac{(-1)^{|I|+1} \int_{w=o}^{z} \omega_{0,2}(z_1,w)}{\prod_{i\in I} \bigl( \omega_{0,1}^{\epsilon}(z)-\omega_{0,1}^{\epsilon}(\sigma_i(z)) \bigr)} \\
			&\qquad
			\times \sum_{\substack{J\vdash I\cup \{r\} \\ \sqcup_{i=1}^{\ell(J)} N_i =\{2,\dots,n\} \\ \sum_{i=1}^{\ell(J)} g_i = g+\ell(J)-|I|-1 }}^{\textup{no } (0,1)}
				\prod_{i=1}^{\ell(J)} \omega_{g_i,|J_i|+|N_i|}^{\epsilon}(\sigma_{J_i}(z),z_{N_i}) \, ,
	\end{split}
	\end{equation}
	where $o\in\P^1$ is an arbitrary base point, and the set $(\sigma_i(z))_{i \in [r]} $ is the set of global sheet involutions -- that depend on $\epsilon$ -- of the degree $ r $ branched covering $ x_{\epsilon} $ such that $\sigma_r(z) = z$. Moreover, as shown in \cite[theorem 5]{BE13}, it is equivalent to the local definition of the topological recursion (\cref{eqn:local:TR} for $ \epsilon  = 0 $ and \cref{eqn:TR} for $ \epsilon \in U\setminus\{0\} $). 

	Second, the first assumption implies that $\mc{S}_{\epsilon}$ is an admissible spectral curve in the sense of \cite[definition~2.7]{BE17} for all $\epsilon\in U$. Namely:
	\begin{itemize}
		\item the Newton polygon of $\mc{S}_{\epsilon}$ has no interior point,

		\item the curve is smooth as an affine curve (however, it is sufficient to check smoothness at the origin $ (x_{\epsilon},y_\epsilon) = (0,0) $).
	\end{itemize}

	By the second point and the fact that global topological recursion is well-defined for all $\mc{S}_{\epsilon}$, we can apply \cite[theorem~3.26]{BE17}. It states that global topological recursion -- \cref{eqn:global:TR} -- on $\mc{S}_{\epsilon}$ is equivalent to the residue formula:
	\begin{equation}\label{eqn:higher:loop:equation} 
		0 =  \sum_{i=1}^{r-1} \Res_{z=\alpha_i} \left( \int_{w=o}^{z} \omega_{0,2}(z_1,w) \right) Q_{g,n}^{\epsilon}(z,z_2,\dots,z_n) \, ,
	\end{equation} 
	where $o \in \P^1$ is an arbitrary base point, and the multidifferential $Q_{g,n}^{\epsilon}$ is defined by
	\begin{equation}\label{eqn:Q}
	\begin{split}
		Q_{g,n}^{\epsilon}(z,z_2,\dots,z_n)
		& =
		\frac{dx_{\epsilon}(z)}{\de P_{\epsilon}/\de y_\epsilon(z)} p_0(z)
		\sum_{k=1}^{r}(-1)^k y_\epsilon(z)^{r-k} \sum_{\substack{I\subset [r] \\ |I|=k}}
			\frac{1}{d x_\epsilon(z)^k} \\
		&\qquad \times 
			\sum_{\substack{ J \vdash I \\ \sqcup_{i=1}^{\ell(J)} N_i = \{2,\dots,n\} \\ \sum_{i=1}^{\ell(J)}g_i = g+\ell(J)-|I|-1 }}
			\prod_{i=1}^{\ell(J)} \omega_{g_i,|J_i|+|N_i|}^{\epsilon} (\sigma_{J_i}(z),z_{N_i}) \, .
	\end{split}
	\end{equation}
	By equivalent, we mean that there exists a unique collection of multidifferentials $ \omega_{g,n} $ on an admissible spectral curve which satisfies the polarization  condition (also known as the projection property)
	\begin{equation}
		\omega_{g,n}(z_1, z_2, \ldots, z_n)
		=
		\sum_{i=1}^{r-1} \Res_{z = \alpha_i} \left( \int_{w=o}^z \omega_{0,2}(z_1,w) \right) \omega_{g,n}(z, z_2, \ldots, z_n)
	\end{equation} 
	and \cref{eqn:higher:loop:equation}, and that this unique solution is constructed by the global topological recursion.

	The object $Q_{g,n}^{\epsilon}(z,z_2,\dots,z_n)$ is a $1$-form in $z$ which, by \cite[lemma~4.7]{BE17}, has poles only at $x_{\epsilon}(z) = x_{\epsilon}(z_j)$, and in particular it has no pole at the ramification points. Thus, upon taking the  limit as $ \epsilon \to 0 $ of \cref{eqn:Q}, we see that the object $ \lim_{\epsilon \to 0} Q_{g,n}^{\epsilon}(z,z_2,\dots,z_n) $ is well-defined and has no poles except at coinciding points where $x(z) = x(z_j)$. Moreover, on the right-hand side we can replace the $ \omega_{g,n}^{\epsilon}( z_1, \ldots, z_n) $ with their limits as $ \epsilon \to 0 $.

	Now, as the unique solution to  \cref{eqn:Q} is given by the global  topological recursion, we have:
	\begin{enumerate}
		\item $\lim_{\epsilon \to 0}\,\omega_{g,n}^{\epsilon}(z_1,\dots,z_n)=\omega_{g,n}(z_1,\dots,z_n)$;

		\item the correlators $\omega_{g,n}(z_1,\dots,z_n)$ satisfy the global topological recursion (\cref{eqn:global:TR}) for $\mc{S}_0$, and therefore also the local topological recursion on the same curve.
	\end{enumerate}
	This concludes the proof.  
\end{proof}

Finally, we are ready to prove the main result of this section (which recovers the results \cite[theorem~7.3]{DNOS19}, \cite[theorem~1.1]{Mil16} and \cite[corollary~4.11]{BCEG21}).

\begin{theorem}\label{thm:limit}
	The CohFT associated to the $r$-Airy spectral curve on $\P^1$, given by
	\begin{equation}
		x(z)=z^r\,,\qquad\quad
		y(z)=z\,, \qquad\quad
		\omega_{0,2}(z_1,z_2)=\frac{dz_1 dz_2}{(z_1-z_2)^2}\,,
	\end{equation}
	is the Witten $r$-spin class $W^{r}_{g,n}$. More precisely, the  Bouchard--Eynard topological recursion correlators -- \cref{eqn:local:TR} -- corresponding to the $r$-Airy spectral curve are
	\begin{equation}\label{TR:corr}
		\omega_{g,n}(z_1,\dots,z_n)
		=
		(-r)^{g-1} \sum_{a_1,\dots,a_n=0}^{r-2} \int_{\overline{\mc{M}}_{g,n}} W^{r}_{g,n}(v_{a_1}\otimes\cdots\otimes v_{a_n})
		\prod\limits_{i=1}^{n} \sum_{k_i\geq 0} \psi_i^{k_i} d\xi^{k_i,a_i}(z_i) \, ,
	\end{equation}
	and the differentials $d\xi^{k,a}(z)$ are given by
	\begin{equation}
		d \xi^{k,a}(z) = (-1)^k\, r\, \frac{\Gamma \bigl(\frac{a+1}{r}+k+1\bigr)}{\Gamma\bigl(\frac{a+1}{r}\bigr)} \, \frac{d z}{z^{kr+a+2}} \,.
	\end{equation}
\end{theorem}

\begin{proof}
	Consider the shifted Witten $r$-spin class $\widetilde{W}^{r,\epsilon}_{g,n}$: for $\epsilon\neq 0$, the correlators $\widetilde{\omega}_{g,n}^{\epsilon}(z_1,\dots,z_n)$ computed by topological recursion from the spectral curve $\widetilde{\mc{S}}_{r,\epsilon}$ satisfy \cref{eqn:correl:shifted:tilde}. 

	On the CohFT side,  we have already seen in \cref{eqn:Witten:limit} that
	\[
		\lim_{\epsilon \to 0} \widetilde{W}^{r,\epsilon}_{g,n}(v_{a_1} \otimes\cdots\otimes v_{a_n})
		=
		W^r_{g,n}(v_{a_1} \otimes\cdots\otimes v_{a_n})\,.
	\]
	Moreover, $\lim_{\epsilon \to 0}\, d \tilde{\xi}^{k,a}(z) = d \xi^{k,a}(z)$. Therefore, the correlators of \cref{eqn:correl:shifted:tilde}  have well-defined limits as $ \epsilon \to 0 $, which are the following:
	\begin{multline}\label{eqn:limit:correlators}
		\lim_{\epsilon \to 0} \, \widetilde{\omega}_{g,n}^{r,\epsilon}(z_1,\dots,z_n) = \\
		=
		(-r)^{g-1} \sum_{a_1,\dots,a_n=0}^{r-2} \int_{\overline{\mc{M}}_{g,n}} W^{r}_{g,n}(v_{a_1}\otimes\cdots\otimes v_{a_n})\prod\limits_{i=1}^{n} \sum_{k_i\geq 0} \psi_i^{k_i} d \xi^{k_i,a_i}(z_i)\,.
	\end{multline}

	On the topological recursion side, the family of spectral curves $\widetilde{\mc{S}}_{r,\epsilon}$ satisfies the assumptions of \cref{prop:limit:tr}. Indeed:
	\begin{enumerate}
		\item for $\epsilon\neq 0$, $\widetilde{\mc{S}}_{r,\epsilon}$ is defined by the algebraic equation:
		\[
			x- y^r+r\,\epsilon \,y=0\,;
		\]

		\item for $\epsilon \neq 0$, the $r-1$ ramification points of $\widetilde{\mc{S}}_{r,\epsilon}$, namely $\epsilon^{\frac{1}{(r-1)}} e^{\frac{2\pi k \iu}{r-1}}$ for $k = 0,\dots,r-2$, are simple and the branch points are distinct. At $\epsilon=0$, the spectral curve $\widetilde{\mc{S}}^{r,0}$ is the $r$-Airy curve, which is admissible;
		
		\item \cref{eqn:limit:correlators} shows that the multidifferentials $\widetilde{\omega}_{g,n}^{r,\epsilon}(z_1,\dots,z_n)$ have well-defined limits as $\epsilon\to 0$.
	\end{enumerate}
	Therefore, we can apply \cref{prop:limit:tr}: for $2g-2+n>0$, local topological recursion on the $r$-Airy spectral curve constructs the correlators
	\[
		\omega_{g,n}(z_1,\dots,z_n)
		=
		(-r)^{g-1} \!\!\!\!\!\!
		\sum_{a_1,\dots,a_n=0}^{r-2} \int_{\overline{\mc{M}}_{g,n}}
			W^{r}_{g,n}(v_{a_1}\otimes\cdots\otimes v_{a_n})
			\prod\limits_{i=1}^{n} \sum_{k_i\geq 0} \psi_i^{k_i} d \xi^{k_i,a_i}(z_i)\,. \qedhere
	\]
\end{proof}

\begin{remark}\label{rmk-limit}
	In principle, the above proof should work in the case of the shifted Witten class $ \widehat{W}^{r,\epsilon}_{g,n} $ as well. However,  the corresponding spectral curve $ \widehat{S}^{r,\epsilon} $ has coinciding branch points (indeed, $ x(\alpha_k) = (-1)^k 2 \epsilon^{r/2}  $), and thus condition (2) in \cref{prop:limit:tr} does not hold. In forthcoming work \cite{BBCKS23}, this condition on coinciding branch points will be relaxed and then the above proof goes through analogously.
\end{remark}

\subsection{Witten's \texorpdfstring{$r$}{r}-spin conjecture}

In 1993, Witten formulated the conjecture \cite{Wit93} that the generating series of $r$-spin intersection numbers satisfy the $r$-KdV hierarchy (or $r$-th Gel'fand--Dikii hierarchy) and the string equation. While the conjecture for $r=2$ has received many proofs by now \cite{Kon92, Mir07+, OP09, ABCGLW20}, the general $r$ case was only proved by Faber--Shadrin--Zvonkine in 2010 \cite{FSZ10}.

The generating series studied in Witten's conjecture is given by
\begin{equation}
	F_{r\textup{-spin}}(\bm{t};\hbar)
	=
	\!\!\! \sum_{\substack{g \geq 0, n \geq 1 \\ 2g-2+n > 0}} \!\!\! \frac{\hbar^{g-1}}{n!} \!\!\!
	\sum_{a_1,\dots,a_n=0}^{r-2}
	\int_{\overline{\mc{M}}_{g,n}} W^{r}_{g,n}(v_{a_1}\otimes\cdots\otimes v_{a_n})
		\prod_{i=1}^{n} \sum_{k_i\geq 0} \psi_i^{k_i} \, t_{k_i,a_i} \, .
\end{equation}
In this section we take advantage of the fact that we have proved topological recursion for $r$-spin intersection numbers without making use of the Faber--Shadrin--Zvonkine result, which allows us to deduce Witten's $r$-spin conjecture from \cref{thm:limit}. 

\begin{theorem}\label{thm:Witten}
	The generating series of $r$-spin intersection numbers $Z_{r\textup{-spin}} = \exp(F_{r\textup{-spin}})$ is the tau function of the $r$-KdV hierarchy satisfying the string equation (unique up to multiplicative constant).
\end{theorem}

\begin{proof}
	Our proof follows from translating the topological recursion for the $r$-Airy spectral curve into the following equivalent settings:
	\begin{enumerate}
		\item The Bouchard--Eynard topological recursion for the $\omega_{g,n}$ from \eqref{TR:corr} is equivalent to the so-called higher abstract loop equations for $\omega_{g,n}$ (see \cite{BE17} or \cite[appendix~C]{BBCCN23}).

		\item The higher abstract loop equations for $\omega_{g,n}$ are equivalent to a system of differential equations \cite[theorem~5.30]{BBCCN23}, which are identified for the $r$-Airy curve with the higher quantum Airy structure associated to the partition function $Z_{(r,r+1)}$ \cite[section~6.1]{BBCCN23}. This higher Airy structure corresponds to specific $W(\mathfrak{gl}_r)$ constraints (we call them $W_r$-constraints here), which determine uniquely its solution $Z_{(r,r+1)}$ up to a multiplicative constant.

		\item In \cite{AvM92}, the authors proved the existence of a tau function $Z_{r\textup{-tau}}$ of the $r$-KdV hierarchy satisfying the string equation. They also showed that the $r$-KdV hierarchy together with the string equation imply the $W_r$-constraints and that the $W_r$-constraints have a unique solution if this exists. Hence they deduce that the $r$-KdV hierarchy and the string equation are equivalent to the $W_r$-constraints. From the previous equivalence, we conclude that $Z_{(r,r+1)} = Z_{r\textup{-tau}}$.

		\item Our \cref{thm:limit} tells us that the partition function associated to the differentials $\omega_{g,n}$ from \eqref{TR:corr}, obtained by the Bouchard--Eynard topological recursion applied to the $r$-Airy spectral curve, is $Z_{r\textup{-spin}}$. The second equivalence gives us $Z_{(r,r+1)} = Z_{r\textup{-spin}}$.
	\end{enumerate}
	We have shown that $Z_{r\textup{-spin}} = Z_{r\textup{-tau}}$, as was needed to recover Witten's $r$-spin conjecture (Faber--Shadrin--Zvonkine theorem).
\end{proof}

\begin{remark}
	In \cite{BCEG21}, the ribbon graphs introduced by Kontsevich in order to prove Witten's conjecture ($r=2$) were generalised to the $r$-spin setting. The generating series of these graphs satisfy topological recursion for the spectral curve from \eqref{eqn:SC:Wspin:tilde}, which corresponds to the shifted Witten class $\widetilde{W}^{r,\epsilon}$, and for its limit $\epsilon  \rightarrow 0$, i.e. the $r$-Airy curve, which corresponds to Witten $r$-spin class. The fact that the $\omega_{g,n}$ behave nicely in the limit was justified in \cite{BCEG21} for that particular family. On the other hand, the generating series of the graphs is related to a generalised Kontsevich matrix model, which was proved to satisfy the $r$-KdV hierarchy and string equation in \cite{AvM92} in the limit $\epsilon \rightarrow 0$. Using this the $r$-spin intersection numbers can be expressed in terms of the graphs \cite[theorem~4.8]{BCEG21}. \Cref{thm:SC:Wspin:tilde} allows us to extend this ELSV-like identification to the shifted class $\widetilde{W}^{r,\epsilon}$.
\end{remark}

\appendix
\section{Exponential integrals}
\label{app:exp:integrals}

The aim of this appendix is to discuss the asymptotic behaviour as $s \to \infty$ of integrals of the form
\begin{equation}
	I_{\Gamma}(s) = \int_{\Gamma} g(z) e^{s h(z)} dz\,,
\end{equation}
where $g$ and $h$ are polynomials and $\Gamma$ is a path in $\C$. In addition, we require that $ h $ has finitely many critical points which are all non-degenerate, and we denote the set of critical points by $ \mathfrak a $.  In order to make sense of the above integral, the integration cycle $\Gamma$ should represent an element of the relative homology $H_1(\C, \C_{-T};\Z)$ for very large $T$, where we set
\begin{equation}
	\C_{-T} = \Set{ z \in \C | \Re(s h(z)) \le - T}\, .
\end{equation}
A basis for the relative homology $H_1(\C, \C_{-T};\Z)$ that is best suited for computations of the asymptotic expansion of $I$ is
given by the Lefschetz thimbles. Such a basis $(\Gamma_a)_{a \in \mathfrak{a}}$ of $H_1(\C, \C_{-T};\Z)$ is indexed by the critical point of $h$, and it is defined as follows:
\begin{equation}
	\Gamma_a
	=
	\Set{
		z(t) | \lim_{t \to -\infty} z(t) = a
	},
\end{equation}
where $z \colon \R \to \C$ is a solution of the steepest descent equation
\begin{equation}
	\frac{dz}{dt} = - \frac{\de \overline{sh}}{\de \overline{z}}\, .
\end{equation}
From the above definition, we conclude that $\Gamma_a$ satisfies two key properties: 1) the real part $\Re(s h)$ decreases on $\Gamma_a$ in directions away from $a$, and 2) the imaginary part $\Im(s h)$ is constant along $\Gamma_a$. Indeed, from the steepest descent equation and the Cauchy--Riemann equation $\de_{z} \overline{sh} = 0$, we deduce that along the flow line
\begin{equation}
	\frac{d \overline{sh}}{dt}
	=
	\frac{\de \overline{sh}}{\de z} \frac{dz}{dt} + \frac{\de \overline{sh}}{\de \overline{z}} \frac{d\overline{z}}{dt}
	=
	- \left| \frac{dz}{dt} \right|^2.
\end{equation}
By taking the real and the imaginary parts of the above equation, we deduce the two key properties:
\begin{equation}\label{eqn:decrease:along:flow}
	\frac{d\Re(sh)}{dt}
	=
	- \left| \frac{dz}{dt} \right|^2
	\le 0 \,,
	\qquad\quad
	\frac{d\Im(sh)}{dt}
	= 0 \,.
\end{equation}
Once we pick an orientation of $\Gamma_a$, it will define a cycle in the relative homology $H_1(\C,\C_{-T};\Z)$ if it is closed. This fails precisely if there is a flow line that starts at $a$ at $t = - \infty$ and ends at another critical point $b \ne a$ at $t = + \infty$. If this is the case, the flow line is called a \emph{Stokes line}. If not, the flow line is called a \emph{Lefschetz thimble}.

Notice that, in order to have a Stokes line connecting the critical points $a$ and $b$, a necessary condition is that the imaginary part of $s h(z)$ coincides at the two critical points: $\Im(s h(a)) = \Im(s h(b))$. In particular, for a generic phase of $s$, this condition is not satisfied and the set of Lefschetz thimbles $(\Gamma_a)_{a \in \mathfrak{a}}$ is a basis of the relative homology $H_1(\C,\C_{-T};\Z)$. The values of $s$ for which a Stokes line occur are also called \emph{Stokes rays}, which divide the complex $s$-plane into open sectors, sometimes called \emph{Stokes sectors}.

The advantage of considering Lefschetz thimbles is that we can easily compute the asymptotic expansion of $I_{\Gamma_a}(s)$ as $|s| \to \infty$ using the steepest descent method, as $I_{\Gamma_a}(s)$ is dominated by the contribution from the critical point $a$. Indeed, choose local coordinates around $a$ such that $h(z) - h(a) = - \frac{\zeta^2}{2}$, so that we have a local change of variables $z = z(\zeta)$. Moreover, consider the Taylor expansion
\begin{equation}
	g(z(\zeta)) dz(\zeta)
	=
	\left(
		\sum_{k \ge 0} c_k \zeta^{2k} 
		+
		\text{odd terms}
	\right) d\zeta\,.
\end{equation}
Then we get the following asymptotic expansion for $|s| \to + \infty$ with $s$ not lying in a Stokes ray:
\begin{equation}\label{asymt:I}
	I_{\Gamma_a}(s)
	\sim
	\sqrt{\frac{2\pi}{s}} e^{s h(a)}
	\sum_{k \ge 0} (2k-1)!! c_k s^{-k}\,,
\end{equation}
with the convention $(-1)!! = 1$. It easy to check that the first coefficient is given by $c_0 = \frac{g(a)}{\sqrt{|h''(a)|}} e^{\iu \vartheta(a)}$, with $\vartheta(a) = \frac{\pi - \arg(h''(a))}{2}$.

Although the asymptotic expansion of $I_{\Gamma_a}(s)$ is independent of the Stokes sectors, the downside of considering such integrals is that the Lefschetz thimble $\Gamma_a$ might jump when $s$ crosses a Stokes ray. In particular, the functions $I_{\Gamma_a}(s)$ are holomorphic functions of $s$ on each Stokes sector, but they might jump when $s$ crosses a Stokes ray.

\begin{example}
	Consider the case of $g(z) = 1$ and $h(z) = \frac{z^3}{3} - z$:
	\begin{equation}\label{eqn:Airy:expl}
		I_{\Gamma}(s) = \int_{\Gamma} e^{s(\frac{z^3}{3} - z)} dz\,.
	\end{equation}
	The critical points of $h$ are given by $a_{\pm} = \pm 1$, so that $h(\pm 1) = \mp \frac{2}{3}$. A necessary condition to have a Stokes ray is that $\Im(s h(1)) = \Im(s h(-1))$, which is equivalent to $s$ being real. In other words, for $s \not\in \R^{\times}$, we have solutions $\Gamma_{\pm}$ to the steepest descent equation starting at $\pm 1$ and forming a base of the relative homology. Moreover, we can actually see the occurrence of a Stokes ray for $s \in \R^{\times}$. For instance, let us consider $s \in \R_+$. The corresponding flow lines are pictured in \cref{fig:Airy:Stokes}, and one can see that a Stokes line starting at $-1$ and ending at $+1$.

	\begin{figure}
   \centering
   \begin{subfigure}[b]{0.4\textwidth}
     \centering
     \includegraphics[width=\textwidth]{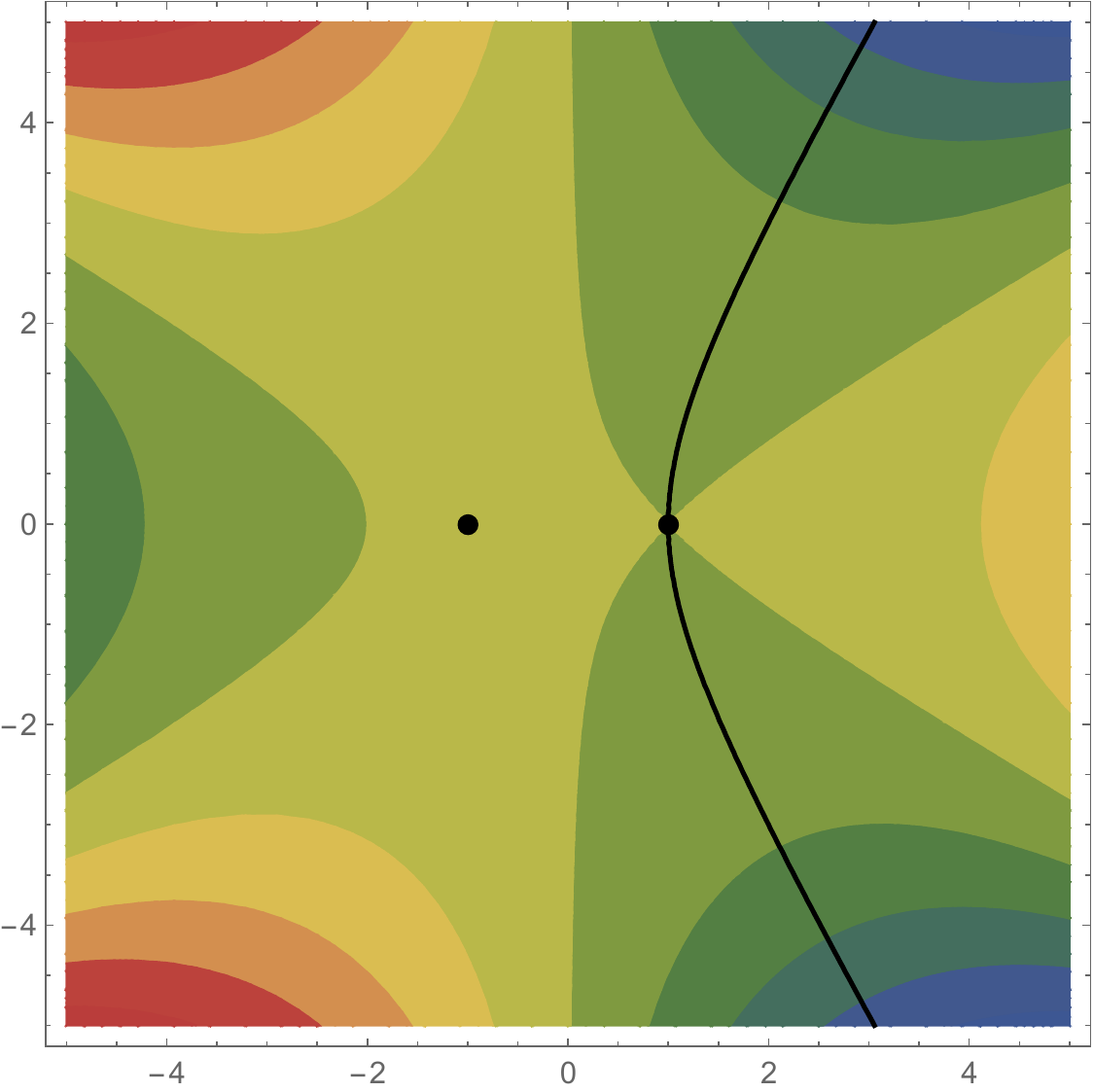}
     \caption{The Lefschetz thimble associated to $z = +1$.}
     \label{fig:Airy:Stokes:0}
   \end{subfigure}
   \hspace{0.15\textwidth}
   \begin{subfigure}[b]{0.4\textwidth}
     \centering
     \includegraphics[width=\textwidth]{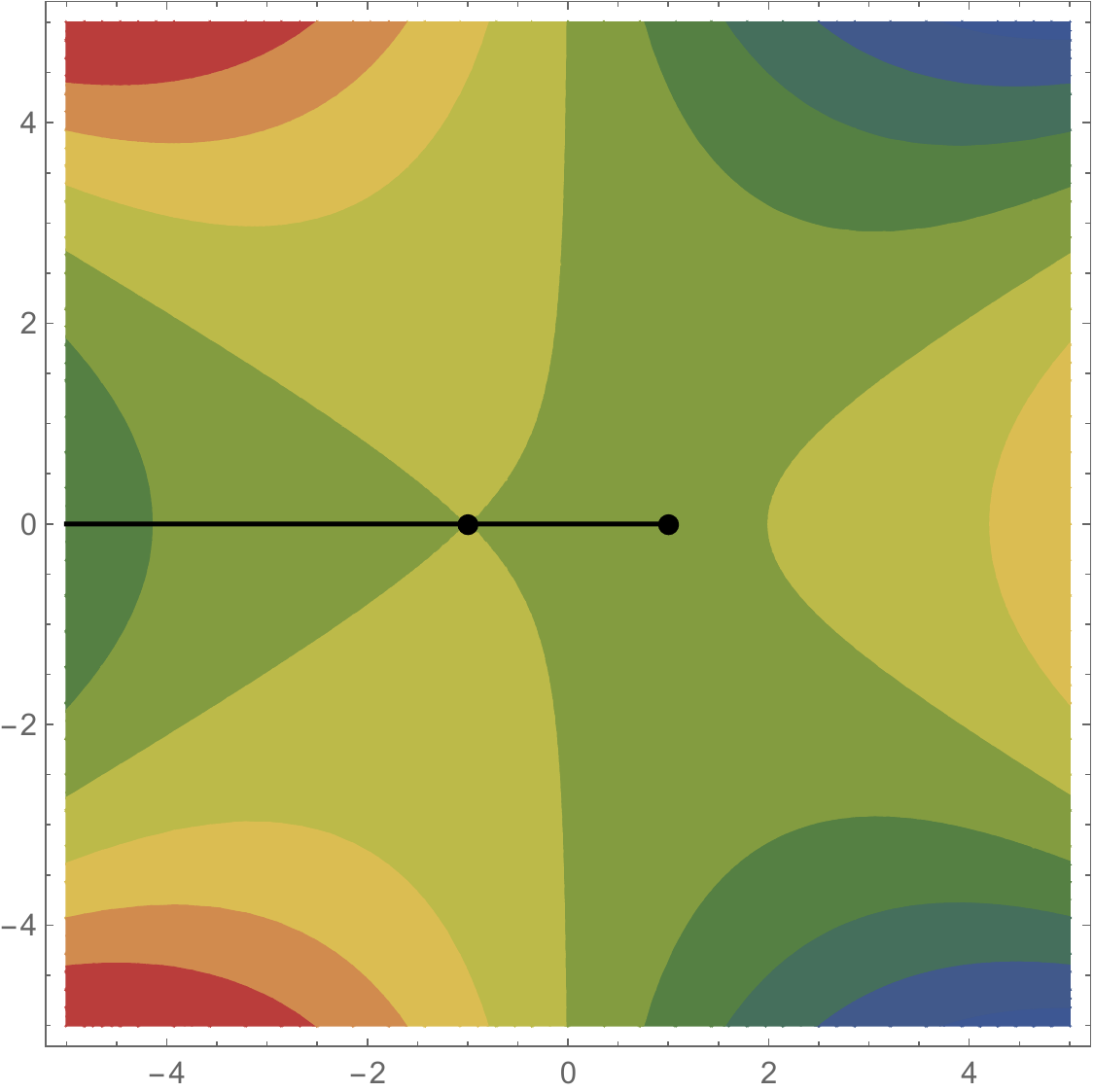}
     \caption{The Stokes line starting at $z = -1$ and ending at $z = +1$.}
     \label{fig:Airy:Stokes:1}
   \end{subfigure}
   \caption{Flow lines for $s = 1$. The colour gradient $[\text{blue},\text{red}]$ corresponds to the values of $\Re(s h) \in [-\infty,+\infty]$.}
   \label{fig:Airy:Stokes}
	\end{figure}

	However, the introduction of a small imaginary part $\epsilon$ to the parameter $s$ will give a well-defined Lefschetz thimble $\Gamma_{-}$. \Cref{fig:Airy:Lefschetz:jump} shows that, depending on the sign of $\epsilon$, the Lefschetz thimble $\Gamma_{-}$ will jump from one side of the negative real axis to the other.

	\begin{figure}
   \centering
   \begin{subfigure}[b]{0.4\textwidth}
     \centering
     \includegraphics[width=\textwidth]{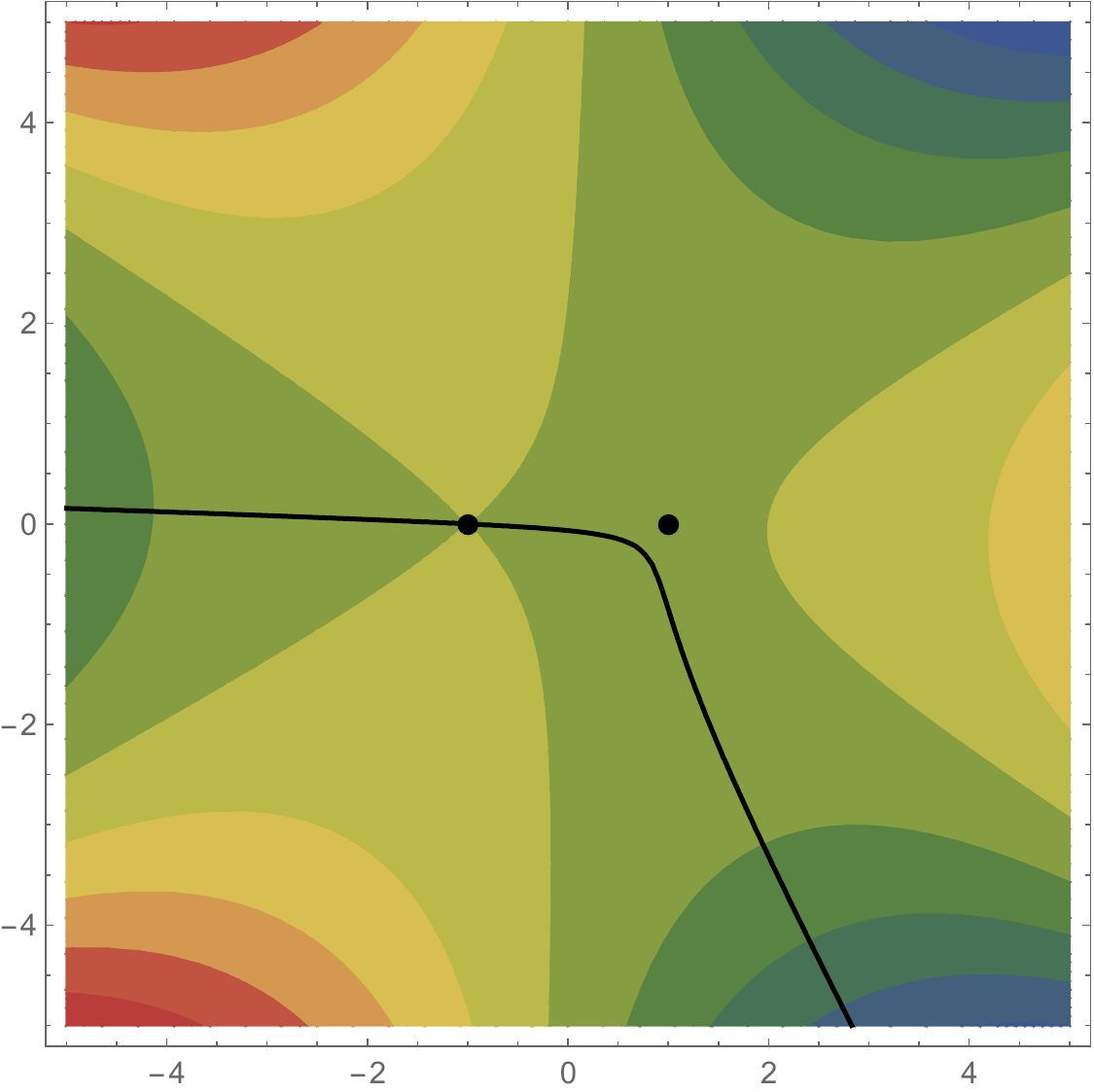}
     \caption{$s = 1 + \iu 0.1$}
     \label{fig:Airy:Lefschetz:jump:pos}
   \end{subfigure}
   \hspace{0.15\textwidth}
   \begin{subfigure}[b]{0.4\textwidth}
     \centering
     \includegraphics[width=\textwidth]{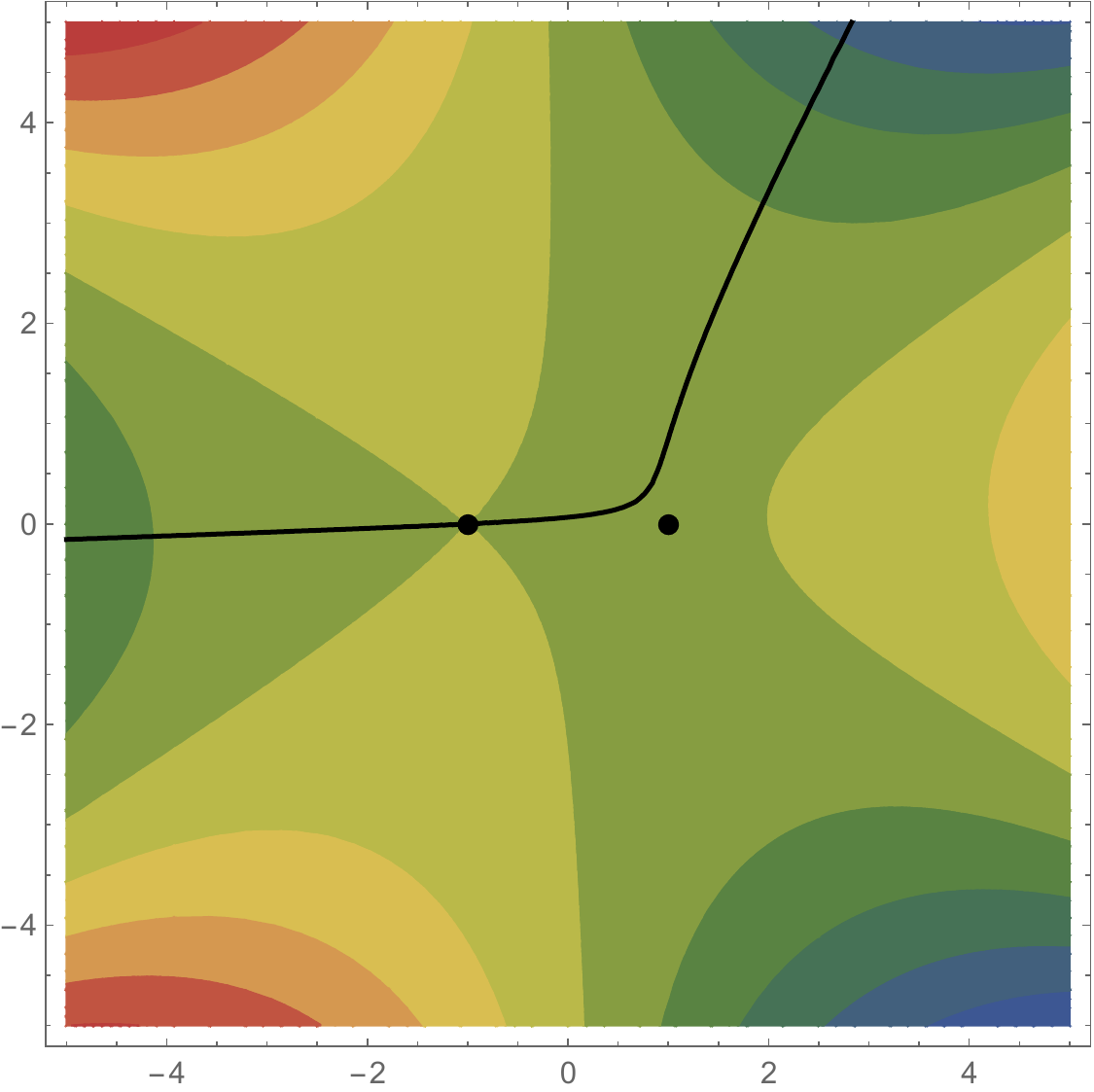}
     \caption{$s = 1 - \iu 0.1$}
     \label{fig:Airy:Lefschetz:jump:neg}
   \end{subfigure}
   \caption{Lefschetz thimbles associated to $z = -1$ for $s = 1 \pm \iu 0.1$. Again, the colour gradient $[\text{blue},\text{red}]$ corresponds to the values of $\Re(s h) \in [-\infty,+\infty]$.}
   \label{fig:Airy:Lefschetz:jump}
	\end{figure}

	The above integral is related to the Airy function as follows. Recall that the Airy function is defined as
	\begin{equation}
		\Ai(t)
		=
		\frac{1}{2\pi\iu} \int_{C} e^{\frac{w^3}{3} - tw} dw\,,
	\end{equation}
	where $C$ is a path starting at $e^{-\iu\frac{\pi}{3}} \infty$ and ending at $e^{\iu\frac{\pi}{3}} \infty$. The change of variables $z = \frac{w}{\sqrt{t}}$ yields
	\begin{equation}
		\Ai(t)
		=
		\frac{t^{1/2}}{2\pi\iu} \int_{\Gamma} e^{t^{3/2}(\frac{z^3}{3} - z)} dz\, ,
	\end{equation}
	where $\Gamma$ is is a path starting at $ e^{-\iu(\frac{\pi}{3} - \frac{\arg(t)}{2})} \infty$ and ending at $e^{\iu(\frac{\pi}{3} - \frac{\arg(t)}{2})} \infty$. The above integral is in the required form, after the identification $s = t^{3/2}$. Notice that the cycle $[\Gamma]$ coincides with $[\Gamma_+]$ in the relative homology group for $|\arg{t}| <\frac{2\pi}{3}$, so that we can recover the asymptotic expansion of the Airy function
	\begin{equation}
		\Ai(t)
		\sim 
		\frac{e^{-\frac{2t^{3/2}}{3}}}{2\sqrt{\pi}} t^{-1/4} \, \sum_{k \ge 0}
 		\frac{(6k)!}{(2k)! (3k)!}
 		\left( -\frac{1}{576 t^{3/2}} \right)^k
	\end{equation}
	from \cref{asymt:I}, valid for $|\arg{t}| <\frac{2\pi}{3}$. A similar analysis can be conducted in the other sectors of the $t$-plane.
\end{example}

\section{On Lucas sequences}
\label{app:Lucas}

In this appendix, we are going to recall some relations between Lucas polynomial sequences $U_n(w,t)$, $V_n(w,t)$ and deduce some expressions for their derivatives. Thanks to such formulae, we can prove a relation that allows us to conclude that the functions \labelcref{eqn:LAiry} satisfy the ODE \labelcref{eqn:LA}.

To begin with, let us recall the recursive definition of Lucas sequences of the first and second kind, denoted respectively as $U_n(w,t)$ and $V_n(w,t)$:
\begin{equation}\label{eqn:rec:Lucas:app}
\begin{split}
	& \begin{cases}
		U_0(w,t) = 0\,, \\
		U_1(w,t) = 1\,, \\
		U_n(w,t) = w U_{n-1}(w,t) - t U_{n-2}(w,t)\,,
	\end{cases}	
	\\
	& \begin{cases}
		V_0(w,t) = 2\,, \\
		V_1(w,t) = w\,, \\
		V_n(w,t) = w V_{n-1}(w,t) - t V_{n-2}(w,t)\,.
	\end{cases}
\end{split}
\end{equation}
Lucas sequences \cite{Luc78} are related to Chebyshev polynomials, as one can easily check comparing the recurrence relations:
\begin{equation}\label{eqn:LtoC}
	U_n(w,t) = t^{\frac{n-1}{2}} \mc{U}_{n-1} \bigl( \tfrac{w}{2\sqrt{t}} \bigr)\,,
	\qquad
	V_n(w,t) = 2 t^{\frac{n}{2}} \mc{T}_{n} \bigl( \tfrac{w}{2\sqrt{t}} \bigr)\,.
\end{equation}
Here $\mc{T}_m$ and $\mc{U}_m$ are the Chebyshev polynomials of the first and second kind respectively. Lucas sequences are also related to various integer sequences like the Fibonacci numbers, Mersenne numbers, Pell numbers, Lucas numbers, and Jacobsthal numbers.

Lucas sequences satisfy many properties, some of which are summarised in the following table.

\begin{table}[h]
\begin{center}
{\renewcommand{\arraystretch}{1.3}%
\begin{tabular}{ l l }
\toprule
	\multirow{2}{*}{Quadratic fomulae} & $U_{n}^2 = t^{n-1} + U_{n+1} U_{n-1}$ \\
	& $V_{n}^2 = - \Delta t^{n-1} + V_{n+1} V_{n-1}$ \\
	\midrule
	\multirow{2}{*}{Conversion fomulae} & $w U_{n} = V_{n} + 2t U_{n-1}$ \\
	& $w V_{n} = \Delta U_{n} + 2t V_{n-1}$ \\
	\midrule
	\multirow{2}{*}{Quasi-homogeneity} & $U_n(a w,a^2 t) = a^{n-1} U_n(w,t)$ \\
	& $V_n(a w,a^2 t) = a^n V_n(w,t)$ \\
\bottomrule
\end{tabular}}
\end{center}
\caption{Some relations satisfied by the Lucas sequences. Here we denote $\Delta = w^2 -4t$, called the discriminant.}
\label{table:Lucas}
\end{table}

Thanks to their relation to Chebyshev polynomials, it is easy to check the following formulae satisfied by the derivatives of the Lucas sequences. In the following, we will denote the derivative with respect to $t$  with a dot and  the derivative with respect to $w$ with a prime.

\begin{lemma}\label{lem:deriv:Lucas}
	Define the discriminant as $\Delta = w^2 -4t$. Then
	\begin{equation}
	\begin{aligned}
		\dot{U}_n & = - \frac{(n-1) V_{n-1} - w U_{n-1}}{\Delta}\,,
		\qquad
		& U_n' & = \frac{n V_{n} - w U_{n}}{\Delta}\,,
		\\
		\dot{V}_n & = - n U_{n-1}\,,
		\qquad
		& V_n'& =  n U_{n}\,.
	\end{aligned}
	\end{equation}
\end{lemma}

The aim of this appendix is to prove the following result.

\begin{proposition}\label{prop:LA:ODE:solution}
	For $a = 0,\dots,r-2$ and $j = 1, \dots, r-1$, the functions
	\begin{equation}
		\hAi_{r,a}^{j}(t) = \frac{1}{2\pi \sqrt{(-1)^j}} \int_{\hat{C}_j} U_{a+1}(w,t) e^{\frac{1}{r} V_r(w,t)} dw
	\end{equation}
	satisfy the differential equation $\ddot{u}(t) = t^{r-2} u(t) + \frac{a}{t} \dot{u}(t)$.
\end{proposition}

\begin{proof}
	Direct computations show that
	\begin{multline*}
		\left( \frac{d^{2}}{dt^{2}} - t^{r-2} - \frac{a}{t} \frac{d}{dt} \right)
		\hAi_{r,a}^{j}(t)
		= \\
		=
		\frac{1}{2\pi \sqrt{(-1)^j}} \int_{\hat{C}_j}
		\Bigl(
			\ddot{U}_{a+1} + 2 \frac{\dot{U}_{a+1} \dot{V}_{r}}{r} + \frac{U_{a+1} \ddot{V}_{r}}{r} + \frac{U_{a+1} (\dot{V}_{r})^2}{r^2}
			\\
			-t^{r-2} U_{a+1}
			-\frac{a}{t} \dot{U}_{a+1} -\frac{a}{tr} U_{a+1} \dot{V}_{r}
		\Bigr) e^{\frac{1}{r} V_r(w,t)} dw\,.
	\end{multline*}
	We now claim that the integrand on the right hand side equals
	\begin{equation*}
		\frac{\partial}{\partial w}\biggl[
			\Bigl(
				U_{a+1} U_{r-2} + \frac{1}{t} U_{a+1}'
			\Bigr) e^{\frac{1}{r} V_r(w,t)}
		\biggr]\,.
	\end{equation*}
	The claim would prove the proposition, since the integral of the above quantity along $\hat{C}_j$ vanishes (by design, $e^{\frac{1}{r} V_r(w,t)}$ is exponentially decreasing along the contour). Moreover, the claim is equivalent to the following relation between Lucas sequences:
	\begin{multline*}
		\ddot{U}_{a+1} + 2 \frac{\dot{U}_{a+1} \dot{V}_{r}}{r} + \frac{U_{a+1} \ddot{V}_{r}}{r} + \frac{U_{a+1} (\dot{V}_{r})^2}{r^2}
		-t^{r-2} U_{a+1}
		-\frac{a}{t} \dot{U}_{a+1} -\frac{a}{tr} U_{a+1} \dot{V}_{r} = \\
		=
		U_{a+1}' U_{r-2} + U_{a+1} U_{r-2}' + \frac{1}{t} U_{a+1}''
		+ \frac{U_{a+1} U_{r-2} V_r'}{r} + \frac{U_{a+1}' V_r'}{tr}\,.
	\end{multline*}
	In order to prove this expression let us group the terms as follows:
	\begin{multline*}
		\textcolor{ForestGreen}{
			\ddot{U}_{a+1}
		}
		\textcolor{Mulberry}{
			+ 2 \frac{\dot{U}_{a+1} \dot{V}_{r}}{r}
		}
		\textcolor{BrickRed}{
			+ \frac{U_{a+1} \ddot{V}_{r}}{r}
		}
		\textcolor{RoyalBlue}{
			+ \frac{U_{a+1} (\dot{V}_{r})^2}{r^2}
			-t^{r-2} U_{a+1}
		}
		\textcolor{ForestGreen}{
			-\frac{a}{t} \dot{U}_{a+1}
		}
		\textcolor{Mulberry}{
			-\frac{a}{tr} U_{a+1} \dot{V}_{r}
		}
		= \\
		=
		\textcolor{Mulberry}{
			U_{a+1}' U_{r-2}
		}
		\textcolor{BrickRed}{
			+ U_{a+1} U_{r-2}'
		}
		\textcolor{ForestGreen}{
			+ \frac{1}{t} U_{a+1}''
		}
		\textcolor{RoyalBlue}{
			+ \frac{U_{a+1} U_{r-2} V_r'}{r}
		}
		\textcolor{Mulberry}{
			+ \frac{U_{a+1}' V_r'}{tr}
		} \,.
	\end{multline*}
	The green terms can be rewritten as
	\[
		\ddot{U}_{a+1}-\frac{a}{t} \dot{U}_{a+1}-\frac{1}{t}U_{a+1}''
		=
		U_{a-1}''+\frac{a}{t}U_{a}'-\frac{1}{t} U_{a+1}''
		=
		\frac{d}{dw} \left(U_{a-1}' + \frac{a}{t}U_{a} - \frac{1}{t} U_{a+1}' \right).
	\]
	We show that the sum of terms to be derived simplifies:
	\begin{equation*}
	\begin{split}
		U_{a-1}' + & \frac{a}{t}U_{a} - \frac{1}{t} U_{a+1}'
		= \\
		& =
		\frac{1}{t\, \Delta} \left( w U_{a+1} - (a+1) V_{a+1}+ a\Delta U_{a} - w t U_{a-1} + (a-1)t V_{a-1} \right) \\
		& =
		\frac{1}{t\, \Delta} \Big[ -a\big(
				\underbrace{w V_{a} - 2t V_{a-1} - \Delta U_{a}}_{=0}
			\big)
			+ w \big(
				\underbrace{w U_{a}-V_{a}-2t U_{a-1}}_{=0}
			\big)\Big] = 0 \, .
	\end{split}
	\end{equation*}
	In the first equality, we use \cref{lem:deriv:Lucas}; in the second one, the recursive definition of Lucas sequences; in the third one, the conversion formulae of \cref{table:Lucas}. To conclude, the green terms simplify. For the purple terms, we notice that
	\[
		U_{r-2} + \frac{V_r'}{tr} = U_{r-2} + \frac{U_{r}}{t} = \frac{w}{t} U_{r-1}\,,
	\]
	using \cref{lem:deriv:Lucas} and the recurrence relation \labelcref{eqn:rec:Lucas:app}. On the other hand,
	\[
	\begin{split}
		2 \dot{U}_{a+1} -\frac{a}{t} U_{a+1}
		&=
		-2\frac{a V_a - w U_a}{\Delta} -\frac{a}{t} U_{a+1} \\
		&=
		- \frac{w}{t} \frac{( V_{a+1} - w U_{a+1} + a V_{a+1})}{\Delta} \\
		&=
		- \frac{w}{t} \frac{(a+1) V_{a+1} - w U_{a+1}}{\Delta}
		= - \frac{w}{t} U_{a+1}'\,.
	\end{split}
	\]
	In the second equality, we used the conversion formulae of \cref{table:Lucas}. Hence, the purple terms simplify. The red terms simplify using \cref{lem:deriv:Lucas}: 
	\[
		U_{a+1} \left( \frac{\ddot{V}_{r}}{r} - U_{r-2}' \right)
		=
		U_{a+1} \bigl( - \dot{U}_{r-1} + \dot{U}_{r-1} \bigr)
		= 0\,.
	\]
	Similarly, the blue expression simplifies as
	\[
		\frac{(\dot{V}_{r})^2}{r^2}
		- t^{r-2}
		- \frac{U_{r-2} V_r'}{r}
		=
		U_{r-1}^2
		- t^{r-2}
		- U_{r-2} U_r
		= 0\,.
	\]
	In the last equality, we used the quadratic formulae of \cref{table:Lucas}. This completes the proof of the above relation between Lucas sequences, hence it proves the proposition.
\end{proof}

In the following lemma, we prove a useful relation among the Lucas sequences.

\begin{lemma}\label{lem:Uroot}
	For any $ k \in \{1,\ldots, r-1\} $, we have the following identity
	\begin{equation}
		\frac{U_r(w,t)}{w - 2 \sqrt{t} \cos\left(\frac{k \pi}{r}\right)}
		=
		\sum_{p=0}^{r-2} (-1)^{k+1} t^{\frac{r-p-2}{2}} \frac{\sin \bigl( \frac{(p+1)k \pi}{r} \bigr)}{\sin \left(  \frac{k \pi}{r} \right)} U_{p+1}(w,t)\,.
	\end{equation}	
\end{lemma}

\begin{proof}
	First, we use the relation between the Lucas sequences and the Chebyshev polynomials~\eqref{eqn:LtoC} to write
	\[
		\frac{U_r(w,t)}{w - 2 \sqrt{t} \cos\left(\frac{k \pi}{r}\right) }
		=
		\frac{t^{\frac{r-2}{2}}}{2} \left( \frac{ \mc{U}_{r-1}(\tfrac{w}{2\sqrt{t}})}{\frac{w}{2\sqrt{t}} - \cos\left(\frac{k \pi}{r}\right)} \right).
	\] 
	As $ z =  \cos\left(\frac{k \pi}{r}\right) $ for any $ k \in \{1,\ldots, r-1\} $ is a root of $ U_{r-1}(z) = 0  $, we have the following expansion
	\begin{equation*}
	 \frac{ \mc{U}_{r-1}(z)}{z - \cos\left(\frac{k \pi}{r}\right)}  = \sum_{p =0}^{r-2} c_p \mc{U}_{p}(z)\,,
	\end{equation*}
	and we would like to determine $c_p$. In the following, we will use various properties of Chebyshev polynomials, and refer to \cite{DLMF}. This can be achieved using the following orthogonality property for the Chebyshev polynomials:
	\[
		\int_{-1}^1 \mc{U}_n(z) \mc{U}_m(z)  \sqrt{1-z^2}   dz  = \frac{\pi}{2} \delta_{n,m}\,.
	\]
	Solving for $c_p$, we find
	\begin{equation*}
	\begin{split}
		c_p
		& =
		\frac{2}{\pi} \int_{-1}^1 \frac{1}{z - \cos\left(\frac{k \pi}{r}\right)} \mc{U}_{r-1}(z) \, \mc{U}_{p}(z) \sqrt{1-z^2} dz \\
		& =
		\sum_{j =0}^{p} \frac{2}{\pi} \int_{-1}^1 \frac{1}{z - \cos\left(\frac{k \pi}{r}\right)} \mc{U}_{r-1 -p +2j}(z) \sqrt{1-z^2} dz\,,
	\end{split}
	\end{equation*}
	where we used the property $\mc{U}_{q}(z)\,\mc{U}_{p}(z)=\sum _{j=0}^{p}\,\mc{U}_{q-p+2j}(z) $ for $ q \geq p $. The integral on the last line can be solved using the formula
	\[
		\int_{-1}^{1} \frac{ \sqrt{1-z^2}\,  \mc{U}_{n-1}(z) }{z-y} dz = -\pi \, \mc{T}_n(y)\,,
	\]
	which gives
	\[
		c_p
		=
		-2 \sum_{j =0}^{p} \mc{T}_{r-p+2j} \left(\cos\left(\frac{k \pi}{r}\right)\right)
		=
		-2 \sum_{j =0}^{p} \cos\left(\frac{(r-p +2j)k \pi}{r}\right).
	\]
	By rewriting the $\cos$ functions as exponentials, and using the formula for a geometric series, we can simplify the $c_p$ to 
	\[
		c_p = (-1)^{k+1} 2 \frac{\sin \bigl( \frac{(p+1)k \pi}{r} \bigr)}{\sin \left( \frac{k \pi}{r} \right)}\,.
	\]
	By using the equation for $c_p$, we finally get
	\begin{equation*}
	\begin{split}
		\frac{U_r(w,t)}{w - 2 \sqrt{t} \cos\left(\frac{k \pi}{r}\right) }
		&=
		\frac{t^{\frac{r-2}{2}}}{2}  \sum_{p =0}^{r-2} (-1)^{k+1} 2 \frac{\sin \bigl( \frac{(p+1)k \pi}{r} \bigr)}{\sin \left(  \frac{k \pi}{r} \right)} \mc{U}_{p}\left(\frac{w}{2\sqrt{t}}\right) \\
		&=
		(-1)^{k+1}  \sum_{p =0}^{r-2} t^{\frac{r-p-2}{2}} \frac{\sin \bigl( \frac{(p+1)k \pi}{r} \bigr)}{\sin \left(  \frac{k \pi}{r} \right)} U_{p+1}\left(w,t\right) \, ,
	\end{split}
	\end{equation*}
	where we rewrote the Chebyshev $\mc{U}_{p}$ polynomials in terms of the Lucas sequences using \cref{eqn:LtoC}. Thus, we get the statement of the lemma.
\end{proof}

Another useful relation is the following.

\begin{lemma}\label{lem:Lid}
	We have the following identity
	\begin{equation}
		\hAi_{r,0}^j(t) = - \frac{1}{2\pi \sqrt{(-1)^j}} \int_{\hat{C}_j} w U_r(w,t) e^{\frac{1}{r}V_r(w,t)} dw\,.
	\end{equation}
\end{lemma}

\begin{proof}
	As
	\[
		\frac{\partial}{\partial w} \left(we^{\frac{1}{r}V_r(w,t)} \right)
		=
		\bigl(1 + w U_r(w,t) \bigr) e^{\frac{1}{r}V_r(w,t)}\,,
	\]
	we obtain that	the integral vanishes:
	\[
		\int_{\hat{C}_j} \bigl( 1 + w U_r(w,t) \bigr) e^{\frac{1}{r}V_r(w,t)} dw = 0\,,
	\]
	due to the definition of the Lefschetz thimble $\hat{C}_j$. Noting that $U_1(w,t) = 1$ gives the statement of the lemma.
\end{proof}


\printbibliography

\end{document}